\documentclass[10pt,a4paper,reqno]{amsart}
\allowdisplaybreaks
\numberwithin{equation}{section}
\usepackage[normalem]{ulem}
\makeatletter 
\def\@tocline#1#2#3#4#5#6#7{\relax
  \ifnum #1>\c@tocdepth 
  \else
    \par \addpenalty\@secpenalty\addvspace{#2}%
    \begingroup \hyphenpenalty\@M
    \@ifempty{#4}{%
      \@tempdima\csname r@tocindent\number#1\endcsname\relax
    }{%
      \@tempdima#4\relax
    }%
    \parindent\z@ \leftskip#3\relax \advance\leftskip\@tempdima\relax
    \rightskip\@pnumwidth plus4em \parfillskip-\@pnumwidth
    #5\leavevmode\hskip-\@tempdima
      \ifcase #1
       \or\or \hskip 1em \or \hskip 2em \else \hskip 3em \fi%
      #6\nobreak\relax
    \dotfill\hbox to\@pnumwidth{\@tocpagenum{#7}}\par
    \nobreak
    \endgroup
  \fi}
\makeatother
\usepackage{transparent}
\usepackage[margin=1.2in]{geometry}
\usepackage{amsmath,amsthm,amssymb}
\usepackage[initials,msc-links]{amsrefs}
\usepackage{url}
\usepackage[dvipsnames]{xcolor}
\usepackage[english=american]{csquotes}
\usepackage{enumerate}
\usepackage{textcomp}
\usepackage{mathrsfs}
\usepackage{mathtools}
\usepackage{color, colortbl}
\definecolor{Gray}{gray}{0.9}
\usepackage{bbm}
\usepackage{hyperref}
\usepackage{stmaryrd}
\usepackage{xfrac}
\hypersetup{
    unicode=false,          
    pdftoolbar=true,        
    pdfmenubar=true,        
    pdffitwindow=false,     
    pdfstartview={FitH},    
    pdftitle={My title},    
    pdfauthor={Author},     
    pdfsubject={Subject},   
    pdfcreator={Creator},   
    pdfproducer={Producer}, 
    pdfkeywords={keyword1, key2, key3}, 
    pdfnewwindow=true,      
    colorlinks=true,       
    linkcolor=blue ,          
    citecolor=blue ,        
    filecolor=magenta,      
    urlcolor=Aquamarine           
}
\usepackage{xcolor}
\usepackage{epsfig,color,graphicx,graphics}
\usepackage{caption}
\usepackage{amsfonts}
\usepackage{tikz}
\usepackage{siunitx}
\usepackage{booktabs,colortbl, array}
\usepackage{pgfplotstable}
\pgfplotsset{compat=1.8}

\definecolor{rulecolor}{RGB}{0,71,171}
\definecolor{tableheadcolor}{gray}{0.92}

\newtheorem{theorem}{Theorem}[section]
\newtheorem{lemma}[theorem]{Lemma}
\newtheorem{proposition}[theorem]{Proposition}
\newtheorem{corollary}[theorem]{Corollary}

\theoremstyle{definition}
\newtheorem{definition}[theorem]{Definition}
\newtheorem{remark}[theorem]{Remark}

\newtheorem{assumption}[theorem]{Assumption}

\newcommand{\Acal}{\mathcal{A}}
\newcommand{\Bcal}{\mathcal B}

\newcommand{\Ecal}{\mathcal{E}}
\newcommand{\Fcal}{\mathcal{F}}
\newcommand{\Gcal}{\mathcal{G}}
\newcommand{\Hcal}{\mathcal{H}}

\newcommand{\Kcal}{\mathcal{K}}
\newcommand{\Lcal}{\mathcal{L}}
\newcommand{\Mcal}{\mathcal{M}}
\newcommand{\Ncal}{\mathcal{N}}

\newcommand{\Ffrak}{\mathfrak{F}}

\newcommand{\Ascr}{\mathscr{A}}

\newcommand{\Bbf}{\mathbf{B}}

\newcommand{\Dbf}{\mathbf{D}}
\newcommand{\Ebf}{\mathbf{E}}

\newcommand{\Gbf}{\mathbf{G}}
\newcommand{\Hbf}{\mathbf{H}}
\newcommand{\Ibf}{\mathbf{I}}
\newcommand{\Jbf}{\mathbf{J}}

\newcommand{\Mbf}{\mathbf{M}}

\newcommand{\Tbf}{\mathbf{T}}

\DeclareMathOperator{\diverg}{div}

\DeclareMathOperator{\dist}{dist}

\newcommand{\N}{\mathbb{N}}
\newcommand{\R}{\mathbb{R}}

\newcommand{\loc}{\mathrm{loc}}

\newcommand{\spt}{\mathrm{spt}}

\newcommand{\Sing}{\mathrm{Sing}}

\newcommand{\mbf}{\boldsymbol{m}}

\newcommand{\eps}{\epsilon}

\DeclareMathOperator{\Err}{Err}

\renewcommand{\eps}{\varepsilon}
\newcommand{\vphi}{\varphi}

\newcommand{\mres}{\mathbin{\vrule height 1.6ex depth 0pt width
        0.13ex\vrule height 0.13ex depth 0pt width 1.3ex}}

\title[]{Structure of two-dimensional mod$(q)$ area-minimizing currents near flat singularities: the codimension one case}

\author[A. Skorobogatova]{Anna Skorobogatova}
\address{Institute for Theoretical Sciences, ETH Z\"{u}rich, 8006 Z\"{urich}, Switzerland}
\email{anna.skorobogatova@eth-its.ethz.ch}

\author[L. Spolaor]{Luca Spolaor}
\address{University of California San Diego, Department of Mathematics, 9500 Gilman Drive \#0112, La Jolla, CA 92093-0112, United States of America}
\email{lspolaor@ucsd.edu}

\author[S. Stuvard]{Salvatore Stuvard}
\address{Università degli Studi di Milano, Dipartimento di Matematica, Via Saldini 50, I-20133 Milano (MI), Italy}
\email{salvatore.stuvard@unimi.it}

\begin{document}

\begin{abstract}
    We obtain a fine structural result for two-dimensional mod$(q)$ area-minimizing currents of codimension one, close to flat singularities. Precisely, we show that, locally around any such singularity, the current is a $C^{1,\alpha}$-perturbation of the graph of a radially homogeneous special multiple-valued function that arises from a superposition of homogeneous harmonic polynomials. Additionally, as a preliminary step towards an analogous result in arbitrary codimension, we prove in general that the set of flat singularities of density $\frac{q}{2}$, where the current is ``genuinely mod$(q)$", consists of isolated points.
\end{abstract}

\maketitle

\section{Introduction}

In this and the companion paper \cite{HSSS}, we study the structural properties of two-dimensional area-minimizing currents mod$(q)$, where $q \geq 2$ is an integer, in a sufficiently regular Riemannian manifold $\Sigma$. Recall that these are two-dimensional integer rectifiable currents in $\Sigma$ which are representatives mod$(q)$ of an area-minimizing flat chain mod$(q)$ in $\Sigma$. Roughly speaking, the framework of currents mod$(q)$ provides a resolution to the Plateau problem with multiplicities in the group $\mathbb{Z}_q$ for both the surface and its boundary; we refer the reader to \cite{DLHMS} for precise definitions and relevant background. In contrast to the case of area-minimizing \emph{integral} currents, namely the case when the coefficients group is $\mathbb Z$, this setting permits the formation of codimension 1 singularities such as triple junctions mod$(3)$ (see \cite{JTaylor76}). Such singularities appear naturally in soap films, and are admissible within the more general framework of stable integral varifolds. Despite many groundbreaking results including \cite{Simon_cylindrical,W14_annals,KW-density-2,Min2,MW} for stable minimal hypersurfaces, the regularity theory in the latter framework is yet to be fully understood, particularly in higher codimension. Nevertheless, the study of mod$(q)$ area-minimizing surfaces provides a valuable insight into what might be expected in general.

The study of mod$(q)$ area-minimizing currents dates back to work of Federer \cite{Federer1970} when $q=2$ (non-oriented surfaces), Taylor \cite{JTaylor76} for two-dimensional surfaces in $\R^3$ when $q=3$, and White \cite{White-regularity,White-mod4} who both studied properties of mod$(4)$ area-minimzing hypersurfaces, and established a regularity result for general mod$(q)$ area-minimizing hypersurfaces. In recent years, a significant number of progress has been made on the regularity of area-minimizing currents mod$(q)$ for general moduli $q$, and in general dimension $m$ and codimension $\bar n$; see \cite{DLHMS_linear,DLHMS,DLHMSS,DLHMSS-fine-structure,DLHMSS-excess-decay,MW,DMS-modp,Liu-mod-p}. Two particularly important outcomes of this regularity theory are
\begin{itemize}
    \item[(1)] at any \emph{classical singularity} that has a tangent cone which is an \emph{open book}, comprised of $q$ (not necessarily distinct) half-planes meeting in an interface, this tangent cone is unique and the {classical singularities} are locally an embedded $C^{1,\alpha}$ $(m-1)$-dimensional submanifold;
    \item[(2)] the set of \emph{branch points} (or \emph{flat singularities}), where there exists a tangent cone supported in a plane with multiplicity forms a countably $(m-2)$-rectifiable set within the $m$-dimensional surface.
\end{itemize}
Note that (1) does not say that the current itself is a $C^{1,\alpha}$ perturbation of the open book locally near a classical singularity; indeed, this is only known if either the codimension of the surface is $\bar n = 1$, or if each of the half-planes in the open book has multiplicity one. Otherwise, one needs to rule out the possibility of branch points accumulating to the spine of the open book. We will treat the latter issue in our companion paper \cite{HSSS} joint with Jonas Hirsch.

Here, we build on this regularity theory, to establish a precise structural characterization of a two-dimensional mod$(q)$ area-minimizing surface in codimension $\bar n = 1$, locally around branch points. This can be viewed as being somewhat analogous to the results in \cite{DLSS1,DLSS2,DLSS3} for two-dimensional integral currents. Note that, for a mod$(q)$ area-minimizing current $T$ of arbitrary dimension and codimension, if a tangent cone to $T$ at a point $p$ is (the current associated with) a plane with density $Q$ strictly smaller than half of the modulus $q$ for $T$, then $T$ identifies with an area-minimizing integral current locally around $p$; see \cite[Proposition 2.7]{DMS-modp}. In particular, when $q$ is an odd integer, branch points are absent in codimension $\bar n = 1$ for any dimension $m$, whereas the local structure of $T$ around any branch point $p$ is understood, after the analysis carried out in \cite{DLSS1,DLSS2,DLSS3}, when the dimension is $m=2$ and the codimension $\bar n$ is higher than $1$. For this reason, we are going to work, here and in \cite{HSSS}, under the assumption that $q$ is an even integer, $q=2Q \geq 4$, and a tangent cone to $T$ at $p$ is a plane with density $Q$. Note that such a planar tangent cone is unique, in light of the work \cite{DLHMSS-excess-decay} in the case $\bar n =1$, and \cite{DMS-modp}*{Theorem 11.5, Proposition 13.3, Theorem 2.6} in the case $\bar n >1$. 

\begin{assumption}\label{ass:main}
    Let $q=2Q \geq 4$ be an integer. Let $\bar n$ and $n$ be integers with $n \ge \bar n \geq 1$, and set $l := n-\bar n$. Let $\kappa \in \left( 0, 1 \right)$. Suppose that $T$ is a 2-dimensional area-minimizing current mod$(q)$ in (an open subset of) a complete, $C^{3,\kappa}$-regular Riemannian $(2+\bar n)$-dimensional embedded submanifold $\Sigma$ of $\R^{2+n}$. Suppose, furthermore, that $0 \in \spt(T) \setminus \spt^q(\partial T)$ is a flat singular point of $T$ with $\Theta(T,0)=Q$, and suppose that $\mathbf{C}=Q\llbracket \pi_0 \rrbracket$ is the tangent cone to $T$ at $0$. Without loss of generality, we will assume that $\pi_0 = \R^2\times\{0\} \subset \R^2\times\R^n$, and that $T_0\Sigma = \R^2\times \R^{\bar n} \times \{0\} \subset \R^2\times\R^{\bar n} \times \R^l$. We shall denote $\pi_0^{\perp_0} = \{0_2\}\times \R^{\bar n}\times \{0_{l}\} \simeq {\R^{\bar n}}$ the orthogonal complement of $\pi_0$ in $T_0\Sigma$, whereas the symbol $^\perp$ will indicate orthogonal complement in $\R^{2+n}$. 
\end{assumption}

The first main result of the present manuscript is in codimension $\bar n=1$, and it can be stated as follows.

\begin{theorem}\label{t:main-structure}
    Let $\bar n=1$, and let $q,Q,n,l,\kappa,T,\Sigma$, and $\pi_0$ be as in Assumption \ref{ass:main}. Then, there exist $r_0>0$ and $\alpha \in \left( 0, 1 \right)$ such that, up to a rotation, $T\mres\Bbf_{r_0}$ is a $C^{1,\alpha}$-perturbation of the multigraph of the function $u:\pi_0 \supset B_{r_0}\to \Ascr_Q(\pi_0^{\perp_0})\simeq\Ascr_Q({\R})$ given in polar coordinates by
    \[
    u(r,\theta):=
    \begin{cases}
    \left(\sum_{i=1}^Q \llbracket c_{j,i}^+\, r^{I_0}\sin(I_0\theta)\rrbracket, +1\right) & \mbox{if $(r,\theta) \in U_j^+$}\,, \\
     \left(\sum_{i=1}^Q \llbracket c_{j,i}^-\, r^{I_0} \sin(I_0\theta)\rrbracket, -1\right) & \mbox{if $(r,\theta) \in U_j^-$}\,,
    \end{cases}
    \]
    where $I_0 \in \mathbb N_{\ge 2}$, $j \in \{0,\ldots,I_0-1\}$, $c_{j,i}^\pm\in\R$, and
    \begin{align*}
        U_j^+ &= \left\{(r,\theta) : 0<r<r_0, \quad \tfrac{2j\pi}{I_0} < \theta \leq \tfrac{(2j+1)\pi}{I_0}\right\}\,, \\ U_j^- &= \left\{(r,\theta) : 0<r<r_0, \quad \tfrac{(2j+1)\pi}{I_0} < \theta \leq \tfrac{(2j+2)\pi}{I_0} \right\}\,.
    \end{align*}
\end{theorem}

Our second main result is that the top density branch points, namely, those that are ``necessarily mod$(q)$", are isolated in any codimension. Let $\Sing_f(T)$ denote the set of flat singularities of a moq$(q)$ area-minimizing current $T$.

\begin{theorem}\label{t:isolated-modq-branch-pts}
    Let $q,Q,n,l,\kappa,T,\Sigma$, and $\pi_0$ be as in Assumption \ref{ass:main}. Then the set 
    \[
        \mathfrak{F}_Q(T) := \{p\in \Sing_f(T): \Theta(T,p) = Q\}
    \]
    is discrete.
\end{theorem}

\begin{remark}\label{r:higher-codim}
    Under the hypotheses of Assumption \ref{ass:main}, even in the case $\bar n > 1$, in Theorem \ref{t:frequency-decay} we are able to obtain uniqueness with a rate of decay of the \emph{fine blow-up} $u$ of $T$ at $0$ (see \cite{DMS-modp}*{Section 11.2}) with the structure given by Theorem \ref{t:main-structure}. However, we are not able to say that $T$ is a $C^{1,\alpha}$-perturbation of the multigraph of $u$ via the methods herein. This is due to the possibility that lower density \emph{classical} branch point singularities may be accumulating to the origin; this will be ruled out in our forthcoming work \cite{HSSS}.
\end{remark}
    
\section*{Acknowledgments}
The authors thank Jonas Hirsch for numerous fruitful discussions. A.S. is grateful for the generous support of Dr. Max R\"ossler, the Walter Haefner Foundation and the ETH Z\"urich Foundation. L.S. acknowledges the support of the NSF Career Grant DMS 2044954. The research of S.S. was supported by project PRIN 2022PJ9EFL "\textit{Geometric Measure Theory: Structure of Singular Measures, Regularity Theory and Applications in the Calculus of Variations}", funded by the European Union under NextGenerationEU and by the Italian Ministry of University and Research.

\section{Strategy of proof and structure of article}
The article is structured as follows. In Section \ref{s:prelim} we provide the relevant notation and preliminary results. In Section \ref{s:tangents} we establish a classification of blow-up functions relative to center manifolds obtained via the blow-up procedure of \cite{DLHMS}, in the case where the dimension of $T$ is two. This is the analogue of the codimension one result \cite{DLHMSS-fine-structure}*{Proposition 2.4}. This allows us to assume that we are working with a single center manifold and $\Mcal$-normal approximation throughout. In Section \ref{s:almost-min} we obtain variational estimates for Almgren's frequency function and related quantities associated to the reparametrization $\Ncal$ of the $\Mcal$-normal approximation $N$. These are analogous to those valid for a regularized version of Almgren's frequency function for $N$. In this section, we additionally verify that $\Ncal$ satisfies a suitable almost-minimizing property for the Dirichlet energy. In Section \ref{s:decay}, we use the almost Dir-minimizing property of $\Ncal$ from the preceding section with a suitable choice of competitor, in order to obtain a decay property for Almgren's frequency function. This in turn induces a power law decay for suitably scaled versions of the $L^2$-height on spheres and the Dirichlet energy on balls for $\Ncal$. From the decay of these 3 quantities, we easily conclude our main result Theorem \ref{t:main-structure} when the codimension is one. 

\section{Notation and preliminaries}\label{s:prelim}
Balls of radius $r$ and center $p$ in $\R^{2+n}$ are denoted by $\Bbf_r(p)$. Two-dimensional planes are denoted by $\pi$, with $\pi_0$ denoting the plane $\R^2\times \{0\}\subset \R^{2\times n}$. We will write $\omega_2$ for the Lebesgue measure of the unit ball in $\R^2$. We use the notation $B_r(p,\pi)$ for the disk $\Bbf_r(p) \cap \pi$ around a point $p\in \pi$. If $p=0$, we will omit dependency on $p$, and when $\pi=\pi_0$, we will often additionally omit dependency on $\pi_0$. We refer the reader to \cite{DLHMS} for the notion of non-oriented tilt excess $\Ebf^{no}(T,\Bbf_r(p), \pi)$ in $\Bbf_r(p)$ relative to a plane $\pi$, and we recall that
\[
    \Ebf^{no}(T,\Bbf_r(p)) = \inf_{\text{planes $\pi$}} \Ebf^{no}(T,\Bbf_r(p), \pi)\,.
\]
Given a point $p\in \R^{2+n}$ and a radius $r>0$ we use the notation $T_{p,r}$ to denote the rescaled current $(\iota_{p,r})_\sharp T$, where $\iota_{p,r}(y) := \frac{y-p}{r}$.

We will henceforth work under the following assumption, which may be achieved by translating and scaling.

\begin{assumption}\label{ass:main+}
	In addition to Assumption \ref{ass:main}, we assume that $\Sigma^{2+\bar n} \subset \R^{2+n}$ is a $C^{3,\kappa}$-regular submanifold with empty boundary in $\mathbf B_{7\sqrt{2}}$, and that, for each $p \in \Sigma$, $\Sigma$ is the graph of a $C^{3,\kappa}$ map $\Psi_p \colon T_p\Sigma \to T_p\Sigma^\perp$. We set $\mathbf{c}(\Sigma):=\sup_{p\in\Sigma\cap\Bbf_{7\sqrt{2}}}\|D\Psi_p\|_{C^{2,\kappa}}$. We additionally assume that $T$ has support in $\Sigma \cap \overline{\mathbf B}_{6\sqrt{2}}$, that it is area-minimizing ${\rm mod}(q)$ in $\Sigma \cap {\mathbf B}_{6\sqrt{2}}$, and that, for some $\varepsilon_0\in\left(0,1\right)$:
    \begin{align}
        \partial T \mres \mathbf B_{6\sqrt{2}} &= 0 \quad {\rm mod}(q) \,, \\
        \|T\|({\mathbf B}_{6\sqrt{2}\rho}) &\leq (Q \omega_2 (6\sqrt{2})^2 + \varepsilon_0^2)\rho^2 \quad \forall\,\rho\leq 1\,,\\
        \mathbf{E}^{no}(T,\mathbf B_{6\sqrt{2}}) &= \mathbf{E}^{no}(T,\mathbf{B}_{6\sqrt{2}},\pi_0)\,,\\
        \mathbf{c}(\Sigma)^2 \leq \varepsilon_0^2\,.
    \end{align}
    Finally, in light of \cite[Lemma 17.7]{DLHMS}, we may assume that $\varepsilon_0$ is so small (depending on $\bar n$ and $n$) that (a suitable modification and extension outside of $\mathbf B_{6\sqrt{2}}$ of) $\Sigma$ is the graph of a $C^{3,\kappa}$ map $\Psi \colon T_0\Sigma \simeq \R^{2+\bar n} \to T_0\Sigma^\perp \simeq \R^l$ satisfying $\Psi(0)=0$ and $\|D\Psi\|_{C^{2,\kappa}} \leq C_0\,\eps_0$ for a geometric constant $C_0=C_0(\bar n,n)$.
    \end{assumption}

We recall the notion of a \emph{center manifold} $\mathcal M$, first introduced by Almgren in \cite{Almgren_regularity} in the framework of area-minimizing integral currents (see also \cite{DLS16centermfld}). Its analogue for area-minimizing currents mod$(q)$ was constructed in \cite{DLHMS} under the validity of Assumption \ref{ass:main+}, and in the context herein, it is a $C^{3,\kappa}$ two-dimensional submanifold of $\Sigma$ obtained as the graph of a map $\vphi$ defined on $B_{3/2} \subset \pi_0$. In the coordinate system defined in Assumption \ref{ass:main}, and denoting $z$ the variable in the plane $\pi_0$, we have $\vphi(z)=(\bar\vphi(z),\Psi(z,\bar\vphi(z))) \in \R^{\bar n} \times \R^l$ for a function $\bar\vphi \colon B_{3/2} \to \R^{\bar n}$, so that when $\mathbf{\Phi}(z)=(z,\vphi(z)) \in \R^2 \times \R^{n}$ denotes the graph map of $\vphi$ then $\mathbf{\Phi}:B_{3/2}\to \Sigma$ and $\Mcal = \mathbf{\Phi}(B_{3/2})$. Following \cite{DLHMS}, we may associate to $\Mcal$ a \emph{normal approximation} for $T$, which is a \emph{special $Q$-valued} Lipschitz map $N: \Mcal \to \Ascr_Q(\R^{2+n})$ that approximates $T$ effectively (see \cite{DLHMS}*{Theorem 17.24, Corollary 17.25}). The latter implicitly comes with a pair $(\Kcal,F)$, where $F$ is the graph map associated with $N$, namely the map $F \colon \Mcal \to \Ascr_Q(\R^{2+n})$ given by 
\[
F(x) = 
\begin{cases}
    \left( \sum_{i=1}^Q \llbracket x + N_i^+(x)\rrbracket, +1 \right) &\mbox{if $x \in \Mcal_+ \cup \Mcal_0$}\,,\\
    \left( \sum_{i=1}^Q \llbracket x + N_i^-(x)\rrbracket, -1 \right) &\mbox{if $x \in \Mcal_-$}\,,
\end{cases}
\]
and $\Kcal\subset \Mcal$ is the (closed) domain of graphicality for which $\Tbf_F\mres \mathbf{p}^{-1}(\Kcal) = T\mres \mathbf{p}^{-1}(\Kcal)$. Here, $x+N_i^\pm(x)$ identifies the point in $\R^{2+n}$ obtained by translating $x\in\Mcal$ according to the vector $N_i^\pm(x)$: these are such that $N_i^\pm(x) \in T_x\Mcal^\perp$, and $x+N_i^\pm(x)\in\Sigma$ for every $i$ and for every $x$. Furthermore, $(\Mcal_+,\Mcal_-,\Mcal_0)$ is the canonical decomposition of $\Mcal$ induced by $N$, $N^\pm$ are the positive and negative parts of $N$, and $\mathbf{p}$ is the nearest point projection map associated to $\Mcal$; see \cite{DLHMSS}*{Assumption 17.21}. We refer the reader to \cite{DLHMS_linear} for the necessary background on special $Q$-valued maps, and to \cite{DLHMS} for all additional relevant terminology and notation surrounding the center manifold in this setting, which we will adopt herein also.

We will further make use of \emph{Almgren's frequency function}, also originally considered in \cite{Almgren_regularity}. 
Recall that for a (special) Dir-minimizing map $u: \Omega \subset \R^m \to \Ascr_Q(\R^d)$ on an open domain $\Omega\subset \R^m$, the frequency of $u$ around a given center $x\in \Omega$ at scale $r\in (0,\dist(x,\partial\Omega))$ is given by
\[
I_u(x,r) := \frac{rD_u(x,r)}{H_u(x,r)}\,,
\]
where
\[
H_u(x,r) = \int_{\partial B_r(x)} \, |u(y)|^2\, d\mathcal{H}^{m-1}(y), \qquad D_u(x,r) = \int_{B_r(x)} |Du|^2(y) \, dy\,.
\]
It is a classical fact that $r\mapsto I_u(x,r)$ is monotone non-decreasing for each $x\in \Omega$, and takes a constant value $\alpha$ if and only if $u$ is radially homogeneous of degree $\alpha$ relative to $x$ (see \cite{DLHMS}*{Section 9} for the derivation of this fact in the $\Ascr_Q$-valued setting). In particular, the limit
\[
	I_u(x,0) := \lim_{r\downarrow 0} I_u(x,r)
\]
exists. 

Following \cite{DLHMS} (cf. \cite{DLS16blowup}), we may subdivide the interval $(0,1]$ of scales around the origin into a collection of mutually disjoint intervals $(s_j,t_j]$ with $t_{j+1} \leq s_j$ for each $j$, referred to as \emph{intervals of flattening}, such that for every $r\in(s_j,t_j]$,
\begin{itemize}
    \item $\Ebf^{no}(T,\Bbf_{6\sqrt{2}{t_j}}) \leq \eps_3^2$,
    \item $\Ebf^{no}(T,\Bbf_{6\sqrt{2}r}) \leq C (\tfrac{r}{t_j})^{2-2\delta_2} \mbf_{0,j}$,
\end{itemize}
where
\begin{equation}\label{e:m0}
    \mbf_{0,j} := \max\{\Ebf^{no}(T,\Bbf_{6\sqrt{2}t_j}), \eps_0^2 t_j^{2-2\delta_2} \},
\end{equation}
and $\eps_3,\delta_2$ are as in \cite{DLHMS}. Note that we may safely take this definition of $\mbf_{0,j}$ in place of the one in \cite{DLHMS}, in light of the scaling of $\boldsymbol{c}(\Sigma_{0,t_j})$ (cf. \cite{DLSk1, DMS-modp}). Associated to each interval $(s_j,t_j]$, we may construct a center manifold $\Mcal_j$ for $T_{0,t_j}\mres\Bbf_{6\sqrt{2}}$ with associated normal approximation $N_j$, as described above. By rescaling, we henceforth work under the following assumption.

\begin{assumption}\label{a:main-m0}
    Suppose that $n$, $\bar n$, $q=2Q$, $T$, $\Sigma$ and $\pi_0$ are as in Assumptions \ref{ass:main} and \ref{ass:main+}. Suppose in addition that $\eps_0$ is chosen small enough so that $t_0=1$ and $\mbf_{0,0}\leq \eps_3^2$.
\end{assumption}

Recalling \cite{DLHMSS-fine-structure}*{Proposition 2.3 \& 2.4} (see also \cite{Sk-modp}*{Proposition 2.3}), when $\bar n=1$, there is a single interval of flattening $(s_0,t_0]=(0,1]$ with corresponding center manifold $\Mcal = \Mcal_0$ and normal approximation $N=N_0$, regardless of the dimension of $T$. The key to this conclusion is the classification of tangent functions and thus possible limiting frequency values at scale 0 given by \cite{DLHMSS-fine-structure}*{Theorem 3.6}. In the succeeding section, we will demonstrate that this remains true when $T$ is two-dimensional, regardless of the codimension $\bar n$, by an analogous classification of tangent functions.

\section{Classification of 2-dimensional tangent functions}\label{s:tangents}

\begin{theorem}\label{t:tangents}
    Suppose that $u\in W^{1,2}(\R^2;\Ascr_Q(\R^{\bar n}))$ is a non-trivial, radially $\alpha$-homogeneous Dir-minimizer with $\boldsymbol{\eta}\circ u = 0$ and such that each of $\Omega^+$, $\Omega^-$ consists of at least one connected component. Let $u^\pm$ be maps associated to the respective positive and negative regions of $u$ (see Section \ref{s:prelim}).

    Then $\alpha$ is a positive integer and the scalar function $|u^+|-|u^-|$ is an $\alpha$-homogeneous harmonic polynomial whose nodal set is $\Sing(u)$.
\end{theorem}

\begin{proof}
    Let $U$ be a connected component of $\Omega^+$. Since $u$ is radially homogeneous, there exists an arc $I\subset \partial B_1$ such that $U = \{t I : t \in (0,\infty)\}$. Now, $u^+|_U$ lies in $W^{1,2}(U;\Acal_Q(\R^{\bar n}))$. Thus, we may use \cite{DLS_MAMS}*{Proposition 1.2} to find a selection $v_1,\dots, v_Q \in W^{1,2}(I,\R^{\bar n})$ such that
    \[
        u^+|_U (r,\theta) = \sum_{i=1}^Q\llbracket r^\alpha v_i(\theta) \rrbracket\,.
    \]
    Moreover, since $\boldsymbol\eta\circ u=0$, we must have $v_1=\cdots=v_Q=0$ on $\partial I$. On the other hand, since $U$ is a single connected component of $\Omega^+$, {there exist $i \neq j$ such that} $v_i(\theta) \neq v_j(\theta)$ for each $\theta \in I$. In addition, the first variations for $u$ guarantee that $v_i$ solve the eigenvalue equation
    \[
        v_i''(\theta) +\alpha^2 v_i(\theta) = 0\,.
    \]
    This has (up to rotation) a unique solution
    \[
        v_i(\theta) = a_i\sin(\alpha \theta)\,,
    \]
    for some $a_i\in \R^{\bar n}$. When combined with the boundary data and the fact that the $v_i$ cannot all agree in the interior of $I$, forces the condition that (up to rotation) $I$ consists of the interval of angles $(0,\tfrac{\pi}{\alpha})$. Since this is true for any such connected component $U$ of $\Omega^+$ or $\Omega^-$, this forces $\alpha \in \N$, and thus $I$ has length $\frac{\pi}{\alpha}$.
    
    To see that $|u^+|-|u^-|$ is an $\alpha$-homogeneous harmonic polynomial, we argue in the same spirit as \cite{DLHMSS-fine-structure}*{Proof of Theorem 3.6}. For each such connected component $U$, let $a(U) = (a_1,\dots, a_Q)$ for $a_i$ as above (which differ among different components $U$).
    Assuming, by rotation, that the connected components $U$ of $\Omega^\pm$ have corresponding arcs $I = \{\tfrac{j\pi}{\alpha} < \theta < \tfrac{(j+1)\pi}{\alpha}\}$, consider the function $p:\R^2 \to \R$ defined in polar coordinates by
    \[
        p(r,\theta) = \begin{cases}
            |a(U)| r^\alpha {|} \sin (\alpha\theta){|} & \text{if $(r,\theta) \in U\subset \Omega^+$} \\
            - |a(U)| r^\alpha {|}\sin (\alpha\theta){|} & \text{if $(r,\theta) \in U\subset \Omega^-$} \\
            0 & \text{otherwise.}
        \end{cases}
    \]
    Observe that $p$ is Lipschitz, $\alpha$-homogeneous and harmonic on $\Omega^+ \cup \Omega^-$. It remains to verify that $p$ is harmonic across $\Omega_0$. To this end, take $U_+, U_-$ connected components of $\Omega^+, \Omega^-$ respectively, such that $\partial U_+ \cap \partial U_- \neq \emptyset$. Let $x\in \partial U_+ \cap \partial U_-{\setminus \{0\}}$. Consider the differential $Dp^\pm(x)$ of $p$ at $x$ from either side. Clearly the tangential derivative $D_\tau p^\pm(x)$ in the direction of $\partial U^\pm$ is zero. On the other hand, taking a tangent function $g$ to $u$ (cf. \cite{DLS_MAMS}*{Section 3.5}, \cite{DLHMS_linear}*{Proof of Proposition 10.3}) at $x$, we observe that $g$ is translation-invariant in the direction spanned by $x$, and thus identifies with a 1-dimensional homogeneous Dir-minimizer $h: \R \to \Ascr_Q(\R^{\bar n})$. 
    Since the inner variation \cite{DLHMS_linear}*{Proposition 7.1} in this case simplifies to
    \[
        \int_\R |Dh|^2 \vphi ' = 0 \qquad \forall\, \vphi \in C_c^\infty(\R)\,,
    \]
    we deduce that $|Dh|$ is constant. Returning to $p^\pm$, we conclude that $|\partial_\nu p^+(x)| = |\partial_\nu p^-(x)|$, where $\nu$ is the outward unit normal to $U_+$ at $x$. By our definition of $p$, we in fact arrive at
    \[
        \partial_\nu p^+(x) = \partial_\nu p^-(x)\,.
    \]
    Since such $U_\pm$ and $x\neq 0$ are arbitrary, this clearly implies that $p$ is indeed harmonic on the entirety of $\R^2 \setminus \{0\}$, and thus on the entirety of $\R^2$.
\end{proof}

By arguing exactly as in \cite{DLHMSS-fine-structure}*{Section 9, Section 10.2}, Theorem \ref{t:tangents} allows us to obtain almost-quadratic (non-oriented) tilt excess decay and conclude that there is a single center manifold such that the flat singularities of $T$ near the origin are contained in its contact set (see \cite{DLHMS} for the notation).

\begin{corollary}\label{c:one-cm}
    Suppose that Assumption \ref{a:main-m0} holds. Then for every $\delta >0$, there exists $\eps_4(\delta,q,n,\bar n)\in (0,\eps_3^2]$ and $C(\delta,q,n,\bar n)>0$ such that if in addition we have $\mbf_0 \leq \eps_4^2$
    then for every $r\in (0,\tfrac{3\sqrt{2}}{16}]$ we have
    \begin{align*}
        \Ebf^{no}(T,\Bbf_r) &\leq C\mbf_0 r^{2-2\delta}\,.
    \end{align*}
    In particular, there is a single interval of flattening $(s_0,t_0]=(0,1]$ and there exists $\eta(q,n,\bar n)>0$ such that 
    \[
        \mathfrak{F}_Q(T) \cap \Bbf_\eta \subset \mathbf\Phi(\Gamma) \subset \Mcal\,,
    \]
    for the associated center manifold $\Mcal=\Mcal_0$.
\end{corollary}

We omit the proof of Corollary \ref{c:one-cm}, since it follows by the very same reasoning as that in the proofs of \cite{DLHMSS-fine-structure}*{Proposition 2.3 \& 2.4}. Note that no part of the proofs therein rely on the codimension being 1, provided that one has the conclusion of Theorem \ref{t:tangents} and the a priori validity of a weaker power law tilt excess decay analogous to \cite{DLHMSS-excess-decay}*{Theorem 1.3}. We note that the latter indeed holds in higher codimension at every point with singularity degree \emph{strictly larger than 1}; see \cite{DLSk1}*{Proposition 7.2} and the discussion in \cite{DMS-modp}*{Section 11.3}. By the classification of Theorem \ref{t:tangents}, when $T$ is two-dimensional, every flat singular point {of density $Q$} has singularity degree at least 2.

In light of this, we will from now on work under the following assumption.

\begin{assumption}\label{a:one-cm}
    Suppose that $n$, $\bar n$, $q=2Q$, $T$, $\Sigma$ and $\pi_0$ are as in Assumption \ref{a:main-m0}. Moreover, suppose that there is a single interval of flattening $(0,1]$ for the origin, with corresponding center manifold $\Mcal$ and $\Mcal$-normal approximation $N$. Let $\mbf_0 := \mbf_{0,0} \leq \eps_4^2$ be the associated constant given by \eqref{e:m0}.
\end{assumption}

\section{Variational estimates and almost Dir-minimality of $\Ncal$}\label{s:almost-min}
In this section, we work under Assumption \ref{a:one-cm} and demonstrate that the $\Mcal$-normal approximation $N$ exhibits a suitable almost Dir-minimizing property relative to Lipschitz competitors (cf. \cite{DLSS3}*{Section 3}. The latter will play a key role in demonstrating, in the succeeding section, a decay property for the frequency function of $N$. Let us first define the latter. Given a center manifold $\Mcal = \mathbf{\Phi}(B_{3/2})$, and a normal approximation $N \colon \Mcal \to \Ascr_Q(\R^{2+n})$, we set $\Ncal := N\circ \mathbf\Phi$. Note that $\Ncal$ provides a parameterization of $N$ over $B_{3/2}\subset \pi_0$. The functions $\Hbf(r):=H_{\Ncal}(0,r)$ and $\Dbf(r) := D_{\Ncal}(0,r)$ are then well defined for every $r \in \left( 0, \frac32\right)$, and the frequency function of $\Ncal$ around $z=0$ at scale $r$ is given by
\[
\Ibf(r) := I_{\Ncal}(0,r) = \frac{r\,\Dbf(r)}{\Hbf(r)}\,,
\]
which is well defined if $\Hbf(r) \neq 0$. {Note that our definition of the frequency function differs from that in \cite{DLS16blowup,DLHMS}. This is entirely for technical purposes, due to the fact that we will derive the key decay property of the frequency function by comparing Fourier decompositions on the sphere $\partial B_r$. We thus wish to work with the classical (non-regularized) Almgren frequency function for $\Ncal$, in place of the regularized frequency function for $N$, but we will proceed to demonstrate that the former satisfies analogous variational identities and estimates.}

Let $((B_{3/2})_+,(B_{3/2})_-,(B_{3/2})_0)$ be the canonical decomposition of $B_{3/2}$ induced by $\Ncal$, and let $\Fcal$ denote the graph of $\Ncal$, namely
\[
    \Fcal(z) := 
    \begin{cases}
    \left( \sum_i \llbracket \mathbf\Phi(z) + \Ncal^+_i(z)\rrbracket, +1 \right) \qquad \mbox{if $z \in (B_{3/2})_+ \cup (B_{3/2})_0$}\,,\\
     \left( \sum_i \llbracket \mathbf\Phi(z) + \Ncal^-_i(z)\rrbracket, -1 \right) \qquad \mbox{if $z \in (B_{3/2})_-$}\,.
     \end{cases}
\]
Throughout the section, all integrals on regions in $\Mcal$ will be intended with respect to $\Hcal^2$, whereas integrals over $B_r$ and $\partial B_r$ in $\pi_0$ will be intended with respect to the Lebesgue measure and arc-length $\Hcal^1$, respectively. The operator $D$ will denote both derivative in $\pi_0$ and tangential derivative on $\Mcal$, depending on the context. Finally, the tangential and normal derivatives on a circle, say, $\partial B_r$ will be denoted $\nabla_\theta$ and $\partial_\nu$, respectively.

\subsection{Variational estimates}
In order to state the key variational estimates for $\Dbf$, $\Hbf$ and $\Ibf$, we introduce the following additional quantities (cf. \cite{DLS16blowup, DLSS3}):
\begin{align*}
    \Gbf(r) &: = \int_{\partial B_r} |\nabla_\theta \Ncal|^2 \,,\\
    \Ebf(r) &:= \sum_{i=1}^Q \int_{\partial B_r}\langle \Ncal_i, \partial_\nu \Ncal_i \rangle\,.
\end{align*}
\begin{proposition}\label{p:variations}
    There exist $\gamma > 0$ and $C=C(q,n,\bar n)$ with the following property. Suppose that Assumption \ref{a:one-cm} holds. Then, the following hold for every $r\in \left(0,1\right]$.
    \begin{align}
        C^{-1} &\leq \Ibf(r) \leq C \label{e:unif_freq} \\
        \Hbf'(r) &= \frac{\Hbf(r)}{r} + 2\Ebf(r) \label{e:H'} \,, \\
        |\Dbf'(r) - 2\Gbf(r)| &\leq C r^{\gamma-1} \int_{\mathbf\Phi(B_r)}|DN|^2 \,, \label{e:inner-var-D'} \\
        |\Dbf(r) - \Ebf(r)| &\leq C r^\gamma \int_{\mathbf\Phi(B_r)} |DN|^2 \,.\label{e:outer-var-D}
    \end{align}
\end{proposition}

Before proving Proposition \ref{p:variations}, let us provide the consequence that the estimates therein have for $\Ibf$.

\begin{corollary}\label{c:freq-monotonicity}
    Suppose that Assumption \ref{a:one-cm} holds. There exists $C= C(q, n, \bar n)$ such that 
    \[
        \Ibf'(r) \geq -C r^{\gamma-1} \qquad \mbox{for every $r\in \left(0,1\right]$}\,,
    \]
    where $\gamma>0$ is as in Proposition \ref{p:variations}.
\end{corollary}
We omit the proof of Corollary \ref{c:freq-monotonicity} here, since given the variational estimates of Proposition \ref{p:variations}, it follows by exactly the same reasoning as that in \cite{DLS16blowup}*{Proof of Theorem 3.2}, combined with the simplification of the frequency radial derivative estimate provided in \cite{DLSk2}*{Lemma 4.1, (31)}.

\begin{proof}[Proof of Proposition \ref{p:variations}]
First, the uniform upper and lower bounds in \eqref{e:unif_freq} are a consequence of a contradiction and compactness argument, cf. \cite{Sk21}*{Proof of Theorem 6.8} and \cite{DLSk2}*{Proof of Lemma 4.1}. Observe that these arguments work analogously in the mod$(q)$ setting, see for instance the discussion in \cite{DMS-modp}*{Section 11}. The identity in \eqref{e:H'} does not require the inner or outer variations for $\Tbf_\Fcal$, and is merely a consequence of directly differentiating $\Hbf(r)$; see for instance \cite{DLS_MAMS}*{Proof of Theorem 3.15}.

We now come to the estimates \eqref{e:outer-var-D} and \eqref{e:inner-var-D'}. These estimates are a consequence of the same computations as in \cite{DLS16blowup}*{Sections 3.1 \& 3.3} and \cite{DLHMS}*{Sections 26.1}, combined with a change of variables to reparameterize to $\pi_0$. Nevertheless, we provide the details here for clarity.

Fix a monotone non-increasing cutoff $\phi\in C^\infty([0,\infty)$ with $\phi \equiv 1$ on $[0,\tfrac{1}{2}]$ and $\phi \equiv 0$ on $[1,\infty)$. We now test the inner and outer variation for the mass of the current $\Tbf_{\Fcal}$ associated to the multigraph $z\mapsto \Fcal(z)$ with the respective vector fields
    \begin{align*}
        X_o(x) &:= \phi\left(\frac{|\mathbf{p}_{\pi_0}(x)|}{r}\right)(x-\mathbf{p}_{\pi_0}(x))\,, \\
        X_i(x) &:= W(\mathbf{p}_{\pi_0}(x)) \,,
    \end{align*}
 with 
 \[
    W(z) :=  \frac{|z|}{r} \phi\left(\frac{|z|}{r}\right) \frac{z}{|z|} \,.
 \]

Let us begin with the outer variation estimate. In light of \cite{DLHMS_linear}*{Theorem 14.2},\footnote{We are applying this result with the domain contained in a plane, thus simplfying the error terms greatly.} we have
    \begin{align*}
        \delta \mathbf{T}_{\Fcal}(X_o) &= \int_{\pi_0} \phi \left(\frac{|z|}{r}\right) |D\Ncal|^2 \\
        &\qquad+ \frac{1}{r}\int_{\pi_0} \sum_j \Ncal_j \otimes \phi' \left(\frac{|z|}{r}\right) \frac{z}{|z|} : D\Ncal_j + O\left(\int_{B_r} |D\Ncal|^3\right)\,.
    \end{align*}
    It therefore remains to control $\delta\mathbf{T}_{\Fcal}(X_o)$. Changing variables, we have 
    \begin{align*}
        \int_{\pi_0} \phi \left(\frac{|z|}{r}\right) |D\Ncal|^2 &= \int_{\Mcal} \phi \left(\frac{|\mathbf\Phi^{-1}(z)|}{r}\right) |DN|^2 |D \mathbf\Phi(\mathbf\Phi^{-1}(z))|^2 \Jbf_{\mathbf\Phi^{-1}}(z) \\
    &= \int_{\Mcal}\phi \left(\frac{|\mathbf\Phi^{-1}(z)|}{r}\right) |DN|^2 + O\left(\mbf_0^{1/2} r \int_{\mathbf\Phi(B_r)} |DN|^2 \right) \,,
    \end{align*}
    and
    \begin{align*}
        \frac{1}{r}\int_{\pi_0} &\sum_j \Ncal_j \otimes \phi' \left(\frac{|z|}{r}\right) \frac{z}{|z|} : D\Ncal_j \\
        &= \frac{1}{r}\int_{\Mcal} \sum_j N_j \otimes \phi' \left(\frac{|\mathbf\Phi^{-1}(z)|}{r}\right) \frac{\mathbf\Phi^{-1}(z)}{|\mathbf\Phi^{-1}(z)|} : DN_j(z) D\mathbf\Phi(\mathbf\Phi^{-1}(z)) \Jbf_{\mathbf\Phi^{-1}}(z) \\
        &= \frac{1}{r} \int_{\Mcal} \phi' \left(\frac{|\mathbf\Phi^{-1}(z)|}{r}\right) \sum_j \left\langle N_j, DN_j \cdot \frac{\mathbf\Phi^{-1}(z)}{|\mathbf\Phi^{-1}(z)|} \right\rangle {+ O\left(\mbf_0 \int_{\mathbf\Phi(B_r)} |N||DN| \right)} \,.
    \end{align*}
    Thus, 
    comparing to $\delta T(Y_o)$ with
    \begin{equation}\label{e:Y_o}
        Y_o(x) = \phi\left(\frac{|\mathbf\Phi^{-1}(\mathbf{p}(x))|}{r}\right)(x-\mathbf{p}(x))\,,
    \end{equation}
    and following the reasoning in \cite{DLS16blowup}*{Sections 3.3 \& 4} and \cite{DLHMS}*{Sections 26.1 \& 26.2}, combined with the simplification of the estimates in \cite{DLSk2}*{Lemma 4.1, (29)},\footnote{Observe that this remains unchanged in the mod$(q)$ setting.} we have 
    \begin{align*}
        |\delta\mathbf{T}_{\Fcal}(X_o)| &\leq \sum_{j=1}^5 |\Err^o_j|  + C \mbf_0^{1/2} r \int_{\mathbf\Phi(B_r)} |DN|^2 + C\mbf_0 \int_{\mathbf\Phi(B_r)} |N||DN| + C\int_{B_r} |D\Ncal|^3\\
        &\leq C r^\gamma \int_{\mathbf\Phi(B_r)} |DN|^2 \,,
    \end{align*}
for some $\gamma>0$, where $\Err_j^o$ are as in \cite{DLHMS}*{(26.9), (26.11)-(26.13)} (cf. \cite{DLS16blowup}*{(3.19), (3.21)-(3.23)}, but with the slightly amended definition of cutoff given in \eqref{e:Y_o}. Notice that in the last estimate we are using \cite{DLS16blowup}*{(3.18)}, namely 
    \begin{equation}\label{e:Poincare-0-avg}
        \int_{\Bcal_r} |N|^2 \leq C r^2 \int_{\Bcal_r} |DN|^2\,,
    \end{equation}
    which remains valid in the mod$(q)$ setting (cf. \cite{DLHMS}*{Proposition 26.4} and the discussion thereafter). Taking a family of such functions $\phi$ approximating $\mathbf{1}_{[0,1]}$ from below and applying the Dominated Convergence Theorem, we conclude \eqref{e:outer-var-D}.

Now let us handle the inner variation estimate. For this, we argue analogously to above, only now we apply \cite{DLHMS_linear}*{Theorem 14.3}. This yields
\begin{align*}
    \delta \Tbf_{\Fcal}(X_i) &= \frac{1}{2} \int_{\pi_0} |D\Ncal|^2 \diverg_{\pi_0} W - \int_{\pi_0} \sum_{j=1}^Q \langle D\Ncal_j : D\Ncal_j \cdot D_{\pi_0} W \rangle \\
    &\qquad+ O\left( \int_{B_r} |D\Ncal|^3\right)\,.
\end{align*}
Changing variables as above, we have 
\begin{align*}
    \frac{1}{2} \int_{\pi_0} |D\Ncal|^2 \diverg_{\pi_0} W &= \frac{1}{2} \int_{\pi_0} |DN(\mathbf\Phi(z))|^2 |D\mathbf\Phi(z)|^2 \diverg_{\pi_0} W(z) \\
    &= \frac{1}{2} \int_{\Mcal} |DN(z)|^2  \diverg_{\pi_0} W(\mathbf\Phi^{-1}(z)) + O\left( \mbf_0^{1/2}\int_{\mathbf\Phi(B_r)} |DN|^2 \right)\,.
\end{align*}
Similarly,
\begin{align*}
    \int_{\pi_0} \sum_{j=1}^Q &\langle D\Ncal_j : D\Ncal_j \cdot D_{\pi_0} W \rangle \\
    &= \int_{\pi_0} \sum_{j=1}^Q \langle DN_j(\mathbf\Phi(z))D\mathbf\Phi(z) : \left( DN_j(\mathbf\Phi(z)) D\mathbf\Phi(z) \cdot D_{\pi_0} W(z) \right) \rangle \\
    &= \int_{\Mcal} \sum_{j=1}^Q \langle DN_j(z) : DN_j(z) \cdot D_{\pi_0} W(\mathbf\Phi^{-1}(z)) \rangle + O\left( \mbf_0^{1/2}\int_{\mathbf\Phi(B_r)} |DN|^2 \right)\,.
\end{align*}
Thus, proceeding as in the outer variation case and comparing $\delta \Tbf_\Fcal(X_i)$ to $\delta T(Y_i)$ for
\begin{equation}\label{e:Y_i}
    Y_i(x) = \frac{|\mathbf\Phi^{-1}(\mathbf{p}(x))|}{r} \phi\left(\frac{|\mathbf\Phi^{-1}(\mathbf{p}(x))|}{r}\right) \frac{\mathbf\Phi^{-1}(\mathbf{p}(x))}{|\mathbf\Phi^{-1}(\mathbf{p}(x))|} \,,
\end{equation}
combined with the simplified estimate in \cite{DLSk2}*{Lemma 4.1, (30)}, we arrive at
\begin{align*}
    |\delta \Tbf_\Fcal(X_i)| &\leq \sum_{j=1}^5 |\Err_j^i| + C \mbf_0^{1/2}\int_{\mathbf\Phi(B_r)} |DN|^2 + C \int_{B_r} |D\Ncal|^3 \\
    &\leq C r^{\gamma-1} \int_{\mathbf\Phi(B_r)}|DN|^2
\end{align*}
for $\gamma>0$ as above, where $\Err_1^i, \Err_2^i, \Err_3^i$ are as in \cite{DLHMS}*{(26.16)-(26.18)} (cf. \cite{DLS16blowup}*{(3.26)-(3.28)}) but with the slightly amended cutoff given in \eqref{e:Y_i}, while $\Err_4^i, \Err_5^i$ are the same as $\Err_4^o, \Err_5^o$ with $X_i$ in place of $X_o$. Once again taking $\phi$ to approximate $\mathbf{1}_{[0,1]}$ from below and applying the Dominated Convergence Theorem, we arrive at \eqref{e:inner-var-D'}.
\end{proof}

\subsection{Competitors and almost Dir-minimality}
Next, we give the definition of Lipschitz competitors for $\Ncal$. 

\begin{definition}
	Given $r\in (0,\tfrac{3}{2})$, we refer to a special $Q$-valued Lipschitz map $\Lcal : B_{r}\to \Ascr_Q(\R^{2+n})$ as a competitor for $\Ncal$ in $B_r$ if
	\begin{itemize}
		\item $\Lcal|_{\partial B_r} = \Ncal|_{\partial B_r}$, and the graph
		\item $\Gcal(z):= \left(\sum_i \llbracket\mathbf\Phi(z) + \Lcal_i(z)\rrbracket, \epsilon(z)\right)$, with $\epsilon(z)=\pm 1$ on $\{|\Lcal^\pm \ominus \boldsymbol\eta \circ\Lcal|>0\}$, is supported in $\Sigma$ for each $z\in B_r\,.$
	\end{itemize}
\end{definition}

Let $\mathbf{p}_0$ denote the orthogonal projection of $\R^{2+n}$ onto $T_0\Sigma=\R^{2+\bar n} \times \{0\} \subset \R^{2+\bar n} \times \R^l$, and consider the Lipschitz special $Q$-valued map $\bar\Ncal :=\mathbf{p}_0 \circ \Ncal$. Notice that it is possible to determine $\Ncal$ from $\bar\Ncal$. Indeed, since $\mathbf{\Phi}(z) + \Ncal_i^\pm(z)\in\Sigma$ for every $z$ and every $i$, we have
\[
\Ncal_i^\pm(z) = \mathbf{p}_0\circ\mathbf{\Phi}(z) + \bar\Ncal_i^\pm(z) + \Psi(\mathbf{p}_0\circ\mathbf{\Phi}(z) + \bar\Ncal_i^\pm(z)) - \mathbf{\Phi}(z)\,.
\]
In the splitting $\R^{2+n}= T_0\Sigma \oplus T_0\Sigma^\perp \simeq \R^{2+\bar n} \times \R^l$, we then have
\begin{equation} \label{e:N-decomp}
\Ncal_i^{\pm}(z) = \left(\bar\Ncal_i^\pm(z),\Psi(\mathbf{p}_0\circ\mathbf{\Phi}(z) + \bar\Ncal_i^\pm(z)) - \Psi(\mathbf{p}_0\circ\mathbf{\Phi}(z))\right)\,.
\end{equation}
Analogously, for a given Lipschitz competitor $\Lcal$ for $\Ncal$ in $B_r$, upon setting $\bar\Lcal := \mathbf p_0 \circ \Lcal$ one obtains that 
\begin{equation} \label{e:L-decomp}
\Lcal_i^{\pm}(z) = \left(\bar\Lcal_i^\pm(z),\Psi(\mathbf{p}_0\circ\mathbf{\Phi}(z) + \bar\Lcal_i^\pm(z)) - \Psi(\mathbf{p}_0\circ\mathbf{\Phi}(z))\right)\,.
\end{equation}
In particular, a Lipschitz competitor for $\Ncal$ in $B_r$ can be always obtained by defining a Lipschitz map $\bar \Lcal \colon B_r \to \Ascr_Q(\R^{2+\bar n})$ such that $\left.\bar\Lcal\right|_{\partial B_r} = \left. \bar\Ncal\right|_{\partial B_r}$ and then using \eqref{e:L-decomp} to lift it to a map $\Lcal$ for which the corresponding graph $\Gcal$ is automatically supported on $\Sigma$.

\medskip

We are now in a position to state the almost Dir-minimizing property for $\Ncal$ relative to Lipschitz competitors. In the following statement, we will use the notation $L:= \Lcal\circ\mathbf\Phi^{-1}$.

\begin{proposition}\label{p:almost-min}
    Suppose that Assumption \ref{a:one-cm} holds. There exists constants $C=C(n, Q)>0$ and $\gamma=\gamma(n,Q)\in (0,1)$ such that the following holds. Suppose that $r\in (0,\tfrac{3}{2})$ and that $\Lcal$ is a competitor for $\Ncal$. Then
    \begin{equation}\label{e:competitor-energy-comparison}
        \Dbf(r) \leq (1+Cr)\int_{B_r} |D\bar \Lcal|^2 + C(\Dbf(r)^\gamma + r)\Dbf(r) + Cr\int_{\mathbf\Phi(B_r)} |\boldsymbol\eta\circ L|\,.
    \end{equation}
\end{proposition}
\begin{proof}
First of all, let us verify that it suffices to demonstrate \eqref{e:competitor-energy-comparison} with $\Lcal$ in place of $\bar \Lcal$. Indeed, by the proof of \cite{DLSS3}*{Lemma A.2},\footnote{Note that the validity of the argument therein, excluding the very last step where Lemma A.1 is applied, is relying merely on the regularity of the parameterizing map $\Psi$ for the ambient manifold $\Sigma$ and thus remains valid in the mod$(q)$ setting.} we have
\[
    \int_{B_r} |D\Lcal|^2 \leq (1+Cr)\int_{B_r} |D\bar \Lcal|^2 + C\int_{B_r} |\bar \Lcal|^2\,.
\]
Since $\bar\Lcal$ is Lipschitz, by a combination of the Fundamental Theorem of Calculus and Fubini's Theorem, we obtain
\[
    \int_{B_r} |\bar \Lcal|^2 \leq Cr^2 \int_{B_r} |D\bar \Lcal|^2 + Cr \int_{\partial B_r} |\bar \Lcal|^2\,.
\]
Recalling that $\bar \Lcal$ shares the values of $\bar \Ncal$ on $\partial B_r$ and combining with the lower bound on the frequency from Corollary \ref{c:freq-monotonicity}, the estimate \eqref{e:competitor-energy-comparison} follows.

    Due to the fact that $\Ncal |_{\partial B_r} = \Lcal |_{\partial B_r}$, the current $T\mres \mathbf{p}^{-1}(\mathbf\Phi(B_r)) - \Tbf_{\Fcal |_{B_r}} + \Tbf_{\Gcal}$ is an admissible competitor for $T$ in $\mathbf{p}^{-1}(\mathbf\Phi(B_r))$ and thus we have
\begin{align*}
    \Mbf(T\mres \mathbf{p}^{-1}(\mathbf\Phi(B_r)) &\leq \Mbf(T\mres \mathbf{p}^{-1}(\mathbf\Phi(B_r)) - \Tbf_{\Fcal |_{B_r}} + \Tbf_{\Gcal}) \\
    &\leq \|T-\Tbf_{\Fcal|_{B_r}}\|(\mathbf{p}^{-1}(\mathbf\Phi(B_r))) + \Mbf(\Tbf_{\Gcal})\,,
\end{align*}
which in turn implies
\begin{equation}\label{e:competitor-area}
    \Mbf(\Tbf_{\Fcal|_{B_r}}) \leq  2\|T-\Tbf_{\Fcal|_{B_r}}\|(\mathbf{p}^{-1}(\mathbf\Phi(B_r))) + \Mbf(\Tbf_{\Gcal})\,.
\end{equation}
We begin by estimating $\|T-\Tbf_{\Fcal|_{B_r}}\|(\mathbf{p}^{-1}(\mathbf\Phi(B_r)))$. First of all, recall that by the regularity properties of $\vphi$ (and hence $\mathbf\Phi$) from \cite{DLHMS}*{Theorem 17.19}, the Jacobian $\Jbf_{\mathbf\Phi}$ of $\mathbf\Phi$ satisfies 
\begin{equation}\label{e:Jacobian}
    \Jbf_{\mathbf\Phi}\leq  1 + C\mbf_0^{1/2}r \qquad \text{on $\mathbf\Phi^{-1}(\Bcal_r)$.}
\end{equation}
We may now perform a decomposition into Whitney regions and use \cite{DLS16blowup}*{Lemma 4.5}, which we observe remains valid in the mod$(q)$ framework in light of the estimates in \cite{DLHMS}, together with \cite{DLHMS}*{(26.7)} and \cite{DLSk2}*{(26)}, as well as \eqref{e:Jacobian} and a change of variables, to obtain
\begin{equation}\label{e:error-from-graphicality}
    \|T-\Tbf_{\Fcal|_{B_r}}\|(\mathbf{p}^{-1}(\mathbf\Phi(B_r))) \leq C\left[\int_{\mathbf{\Phi}(B_r)}|DN|^2\right]^{1+\gamma} \leq C\left[\int_{B_r} |D\Ncal|^2\right]^{1+\gamma}\,,
\end{equation}
for a suitable geometric constant $\gamma=\gamma(n,Q) >0$.

Now, we may perform a curvilinear Taylor expansion of the area using \cite{DLHMS_linear}*{Theorem 13.1}:
\begin{align*}
    \int_{\mathbf\Phi (B_r)}|DN|^2 &\leq 2\Mbf(\Tbf_{\Fcal |_{B_r}}) -2Q\Hcal^m(\mathbf\Phi(B_r)) + C\int_{\mathbf\Phi(B_r)}(|A_\Mcal|^2|N|^2 + |DN|^4) \\
    &\qquad+ 2Q\int_{\mathbf\Phi(B_r)} \langle \boldsymbol\eta \circ N, H_\Mcal\rangle\,.
\end{align*}
Combining the above estimate with \eqref{e:competitor-area}, \eqref{e:error-from-graphicality} and another change of variables, we have
\begin{align*}
    \Dbf(r) &= \int_{B_r} |DN(\mathbf\Phi(z))|^2 |D\mathbf\Phi(z)|^2 \\
    &= \int_{\mathbf\Phi(B_r)}|DN(x)|^2 |D\mathbf\Phi(\mathbf\Phi^{-1}(x))|^2 \Jbf_{\mathbf\Phi^{-1}}(x) \\
    &\leq \int_{\mathbf\Phi(B_r)}|DN|^2 + C\mbf_0^{1/2} r \Dbf(r) \\
    &\leq C(\Dbf(r)^\gamma + r)\Dbf(r) +2\Mbf(\Tbf_{\Gcal|_{B_r}}) -2Q\Hcal^m(\mathbf\Phi(B_r)) \\
    &\qquad+ C\int_{\mathbf\Phi(B_r)}(|A_\Mcal|^2|N|^2 + |DN|^4) + 2Q\int_{\mathbf\Phi(B_r)} \langle \boldsymbol\eta \circ N, H_\Mcal\rangle\,.
\end{align*}
Let us now Taylor expand $\Tbf_{\Gcal}$. This time, we wish to be more careful with the error terms, since $\Lcal$ is a general competitor. Let $e_1,e_2$ be an orthonormal frame on $\Mcal$
\begin{align*}
    \Mbf(\Tbf_{\Gcal}) &\leq \frac{1}{2} \sum_{i=1}^Q\int_{\mathbf\Phi(B_r)} |e_1 + DL_i \cdot e_1|^2 + |e_2 + DL_i \cdot e_2|^2 \\
    &= Q\Hcal^m(\mathbf\Phi(B_r)) + \frac{1}{2} \int_{\mathbf\Phi(B_r)} |DL|^2 + \sum_{i=1}^Q\sum_{j=1}^2 D_{e_j}(\langle L_i, e_j\rangle) - \sum_{i=1}^Q\sum_{j=1}^2 \langle L_i, D_{e_j} e_j\rangle \\
    &= Q\Hcal^m(\mathbf\Phi(B_r)) + \frac{1}{2} \int_{\mathbf\Phi(B_r)} |DL|^2 - Q\int_{\mathbf\Phi(B_r)} \langle \boldsymbol\eta \circ L, H_\Mcal\rangle \,.
\end{align*}
Inserting this into the preceding estimate, and again using a change of variables, we deduce that
\begin{align}
    \Dbf(r) &\leq \int_{\mathbf\Phi(B_r)} |DL|^2 + C(\Dbf(r)^\gamma + r)\Dbf(r) + C\int_{\mathbf\Phi(B_r)}(|A_\Mcal|^2|N|^2 + |DN|^4) \label{e:final-energy-comparison} \\
    &\qquad + 2Q\int_{\mathbf\Phi(B_r)} \langle \boldsymbol\eta \circ N - \boldsymbol\eta\circ L, H_\Mcal\rangle \notag \\
    &\leq (1+Cr)\int_{B_r} |D\Lcal|^2 + C(\Dbf(r)^\gamma + r)\Dbf(r) + C\int_{\mathbf\Phi(B_r)}(|A_\Mcal|^2|N|^2 + |DN|^4) \notag \\
    &\qquad + 2Q\int_{\mathbf\Phi(B_r)} \langle \boldsymbol\eta \circ N - \boldsymbol\eta\circ L, H_\Mcal\rangle\,.\notag
\end{align}
First of all, \cite{DLHMS}*{Corollary 17.25} and \cite{DLS16blowup}*{Lemma 4.5} tell us that, up to decreasing $\gamma$,
\begin{equation}\label{e:error-higher-order-Dir}
    \int_{\mathbf\Phi(B_r)}|DN|^4 \leq C \Dbf(r)^{1+\gamma}
\end{equation}

Furthermore, in light of \eqref{e:Poincare-0-avg} combined with the uniform lower bound on $\Ibf$ given by Corollary \ref{c:freq-monotonicity}, we have
\begin{equation}\label{e:cm-curvature-error}
    \int_{\mathbf\Phi(B_r)} |A_\Mcal|^2|N|^2 \leq C \int_\Mcal \phi\left(\frac{d(y)}{C_0 r}\right) |DN|^2 \leq C \Dbf(r)\,,
\end{equation}
for some $C, C_0 >0$, where $d(y)$ denotes the geodesic distance between $y$ and $\mathbf{\Phi}(0)$ on $\Mcal$.

It thus remains to estimate the average of the sheets of $N$. To that aim, we may simply exploit the estimates \cite{DLHMS}*{(17.25), (17.34)} and \cite{DLS16blowup}*{Lemma 4.5}, together with another application of \cite{DLHMS}*{Theorem 17.19} which in particular gives control on $H_\Mcal$, to conclude that
\begin{equation}\label{e:avg-N}
    \int_{\mathbf\Phi(B_r)} |\langle \boldsymbol\eta\circ N, H_\Mcal\rangle| \leq Cr\int_{\mathbf\Phi(B_r)} |\boldsymbol\eta\circ N| \leq Cr\Dbf(r)^{1+\gamma}\,,
\end{equation}
where we once again decrease $\gamma$ to a smaller geometric constant if necessary. Collecting together \eqref{e:final-energy-comparison} and the error estimates \eqref{e:cm-curvature-error}-\eqref{e:avg-N}, the desired bound \eqref{e:competitor-energy-comparison} follows immediately.
\end{proof}

\section{Choice of competitor and decay of the frequency}\label{s:decay}
We are now in a position to introduce a suitable competitor for $\Ncal$. We begin by setting up some notation. Let $\Omega_{j}^\pm$, $j\in\N$, denote the connected components, ordered from largest to smallest size, of the sets $\Omega^\pm:= \{|\Ncal^\pm| > 0\}\subset B_{3/2}$, where $\Ncal^\pm$ are defined as in \cite{DLHMSS}.  Let
\[
    \alpha := \Ibf(0) = \lim_{r\downarrow 0} \Ibf(r)\,,
\]
and for $j=1,\dots,\alpha$, let $\Ncal^{j,\pm}_1,  \dots , \Ncal^{j,\pm}_Q$ be the sheets of $\Ncal$ in $\Omega_j^\pm$.\footnote{When $\bar n = 1$, the sheets may be ordered, but we do not require this information.}

Consider the rescalings
\[
    \Ncal_r(x) := \frac{(N\circ \mathbf\Phi)(rx)}{\Dbf(r)^{1/2}}\,.
\]

In this section, we will work under the following assumption, which will later be ensured, up to rotating, relabelling indices and rescaling, by \cite{DLHMSS-fine-structure}*{Theorem 3.6, Proposition 2.8} applied to the rescalings $\Ncal_r$, together with a pinching hypothesis on the frequency function close to its value at zero scale.\footnote{Note that in our setting, $\mathbf\Phi$ can be identified with the exponential map for $\Mcal$ centered at 0 in the obvious way.}

\begin{assumption}\label{a:close-to-tangent}
Let $\eta>0$ and $r>0$. Suppose that for $j=1,\dots, \alpha$ and for $F_j^\pm := \Omega^\pm_j\cap\partial B_{r}$ we have
\begin{align*}
   &\partial B_r\cap \left\{(2j-2)\frac{\pi}{\alpha} + \eta < \theta < (2j-1)\frac{\pi}{\alpha} - \eta\right\} \subset F_{j}^+\,, \\
	&F_j^+ \subset \partial B_r\cap\left\{(2j-2)\frac{\pi}{\alpha} - \eta < \theta < (2j-1)\frac{\pi}{\alpha} + \eta\right\}\,,
\end{align*}
and
\begin{align*}
  &\partial B_r\cap\left\{(2j-1)\frac{\pi}{\alpha} + \eta < \theta < 2j\frac{\pi}{\alpha} - \eta \right\} \subset  F_{j}^-\,, \\
  &F_j^-\subset \partial B_r\cap\left\{(2j-1)\frac{\pi}{\alpha} - \eta < \theta < 2j\frac{\pi}{\alpha} + \eta \right\}\,,
\end{align*}
where $\theta \in [0, 2\pi)$ denotes the angular variable in $\pi_0$. Meanwhile, the other connected components $\Omega_j^\pm$, $j\geq \alpha + 1$, of $\Omega^\pm$ satisfy the property
\[
    \Omega_{j}^+ \cap \bar{B}_{r} \subset \bigcup_{k=1}^{2\alpha}\left\{x: \dist(x, L_k) < c{r}\eta \right\},
\]
where $L_k$ is the half-line $\{\theta = \frac{k\pi}{\alpha}\}$ and $c$ is a suitable dimensional constant.
\end{assumption}

Under Assumption \ref{a:close-to-tangent}, we have the following properties for the Laplace eigenvalues on $r^{-1}F_j^\pm \subset \partial B_1$. For each $j$, let $\{\vphi_{k}^{j,\pm}\}_{k\geq 1}$ denote an orthonormal basis of Dirichlet eigenfunctions for the Laplace-Beltrami operator on $r^{-1}F_{j}^\pm$, with corresponding eigenvalues $\lambda_{k}^{j,\pm}$, ordered such that $\lambda_{k}^{j,\pm} \leq \lambda_{k+1}^{j,\pm}$. In particular, there is a unique, non-negative, unit-length eigenfunction $\vphi_{1}^{j,\pm}$ with eigenvalue $\lambda_1^{j,\pm}$. Moreover, for any $\varsigma >0$, there exists $\eta>0$ sufficiently small such that we have
\begin{itemize}
    \item[(i)] $|\lambda_1^{j,
    \pm} - \alpha^2| \leq \varsigma$ \quad for $j=1,\dots,\alpha$,
    \item[(ii)] $\lambda_k^{j,\pm} \geq (\alpha+1)^2 - \varsigma$ \quad for $k\geq 2$ and $j=1,\dots,\alpha$,
    \item[(iii)] $\lambda_k^{j,\pm} \geq {\varsigma^{-1}}$ \quad for $j\geq \alpha +1$.
\end{itemize}
Let $U_{j}^\pm$ denote the cones
\[
    U_{j}^\pm := \{t F_{j}^\pm : t\in (0,1]\}.
\]

\subsection{Construction of competitor}
We now proceed to construct a competitor $\Lcal$. Recalling \eqref{e:L-decomp}, we may do this by constructing $\bar \Lcal$ and then lifting to the desired map $\Lcal$ whose graph takes values in $\Sigma$. We begin by constructing it in the annulus $B_r\setminus B_{\sigma r}$ for a (small) fixed parameter $\sigma\in(0,1)$ to be determined. We start with the portions of the annulus intersecting the {small} cones $U_j^{\pm}$. Fix $j\geq \alpha + 1$, a sign $\epsilon \in \{+,-\}$, and an index $\ell\in \{1,\dots, 2+\bar n\}$ representing a given coordinate of $\bar \Ncal_i$ in $T_0 \Sigma$. To simplify notation, we omit depedency on $j$, $\epsilon$ and $\ell$. For each $i=1,\dots, Q$ and each coordinate , let $\bar \Lcal_i$ be the harmonic function in $V= U\cap(B_r\setminus \bar{B}_{\sigma r})$ satisfying
\[
    \begin{cases*}
        \bar\Lcal_i = \bar\Ncal_i & \text{on $F$,} \\
        \bar\Lcal_i = 0 & \text{on the rest of $\partial(U \cap (B_r\setminus B_{\sigma r}))$.}
    \end{cases*}
\]
Namely, if in polar coordinates we have
\[
    \bar\Ncal_i(r,\theta) = \sum_{k\geq 1} a_{i,k} r^{\alpha_k^+}\vphi_k(\theta) \qquad \text{on $F$,}
\]
then
\[
    \bar\Lcal_i(\rho,\theta) = \sum_{k\geq 1} a_{i,k} (A_k \rho^{\alpha_k^+} - B_k \rho^{\alpha_k^-}) \vphi_k(\theta), \qquad \rho\in (0,r]\,,
\]
for $\alpha_k^\pm = \pm \sqrt{\lambda_k}$ and $A_k, B_k$ explicitly defined by 
\begin{equation}\label{e:coeffs-annulus}
    A_k = \frac{\sigma^{\alpha_k^-}}{\sigma^{\alpha_k^-} - \sigma^{\alpha_k^+}}, \qquad B_k = \frac{\sigma^{\alpha_k^+}}{\sigma^{\alpha_k^-} - \sigma^{\alpha_k^+}} {r^{2\alpha_k^+}},
\end{equation}
so that we satisfy the boundary conditions $A_k r^{\alpha_k^+} - B_k r^{\alpha_k^-} = r^{\alpha_k^+}$ and $A_k (\sigma r)^{\alpha_k^+} = B_k (\sigma r)^{\alpha_k^-}$.

We now define $\bar\Lcal_i$ in the intersection of the annulus and the large cones $U_j^\pm$. Fix $j=1,\dots,\alpha$ and a sign $\epsilon\in\{+,-\}$ and again omit dependency on $j$,$\epsilon$, and the index $\ell \in\{1,\ldots,2+\bar n\}$ representing the component of $\bar\Ncal_i$. This time, again writing
\[
    \bar\Ncal_i(r,\theta) = \sum_{k\geq 1} a_{i,k} r^{\alpha_k^+}\vphi_k(\theta) \qquad \text{on $F$,}
\]
we define the extension $\bar\Lcal_i$ to $V=U\cap (B_r\setminus \bar{B}_{\sigma r})$ by
\[
    \bar\Lcal_i(\rho,\theta) = a_{i,1}  \rho^{\alpha_1^+} \vphi_1(\theta) + \sum_{k\geq 2} a_{i,k} (A_k \rho^{\alpha_k^+} - B_k \rho^{\alpha_k^-}) \vphi_k(\theta)\,,
\]
for $A_k,B_k$ as in \eqref{e:coeffs-annulus} and $\alpha_k^\pm = \pm \sqrt{\lambda_k}$ again.

Finally, in $B_{\sigma r}\cap U_j^\epsilon$ for $j =1,\dots,\alpha$ and $\epsilon\in \{+,-\}$, we let $\bar\Lcal_i$ be the $\alpha_1^{j,\epsilon,+}$-homogeneous radial extension of its boundary data on $\partial B_{\sigma r}\cap U_j^\epsilon$, extending by zero in $B_{r\sigma} \setminus \bigcup_{j=1}^\alpha \bigcup_{\epsilon\in \{+,-\}} U_j^\epsilon$.

We may now define $\Lcal$ as follows. The above procedure defines each coordinate and each of the $Q$ sheets of {classical $Q$-valued functions} $\bar\Lcal^{{\pm}} : B_r \to {\mathcal A}_Q(\R^{2+\bar n})$. {Observe that on each $F_j^\pm$ one has $\bar\Ncal^{\mp}=Q\llbracket \boldsymbol \eta \circ \bar\Ncal \rrbracket$, which implies that on each $U_j^\pm$ one has $\bar\Lcal^\mp=Q\llbracket \boldsymbol\eta\circ \bar\Lcal \rrbracket$. Hence, a special $Q$-valued function $\bar\Lcal \colon B_r \to \Ascr_Q(\R^{2+\bar n})$ is well defined so that the induced canonical decomposition of $B_r$ satisfies $(B_r)_{\pm} = \bigcup_j U_j^\pm$}. We then define $\Lcal$ in terms of $\bar\Lcal$ using the formula \eqref{e:L-decomp}.

\subsection{Frequency decay and proof of main results}
We are now in a position to demonstrate the following key estimate for the frequency function, from which the conclusion of Theorem \ref{t:main-structure} and Theorem \ref{t:isolated-modq-branch-pts} follows easily.

\begin{theorem}\label{t:frequency-decay}
    Suppose that $T$ is as in Assumption \ref{a:one-cm} and suppose that $\Ibf(0) = \alpha \in \N_{\geq 2}$. There exists $\bar\delta(\alpha) > 0$, $\gamma(\alpha, n, Q) >0$ and $C(\alpha, n,Q)>0$ such that the following holds. Suppose that $r_0\in (0,\tfrac{3}{2})$ is such that $|\Ibf(r_0) - \alpha| < \bar\delta$. Then {there exists $r_1=r_1(r_0) \in (0,r_0]$ such that}
    \begin{equation}\label{e:freq-decay}
        |\Ibf(r) - \alpha| \leq C\, r^\gamma \qquad \forall r \in (0,r_1]\,.
    \end{equation}
    Moreover, there exist constants $\Hbf_0>0$ and $\Dbf_0>0$ such that for every $r\in (0,r_1]$ we have
    \begin{align}
        \left|\frac{\Hbf(r)}{r^{2\alpha+1}} - \Hbf_0\right| &\leq C r^\gamma\,,\label{e:H-decay}\\
        \left|\frac{\Dbf(r)}{r^{2\alpha}} - \Dbf_0\right| &\leq C r^\gamma\,.\label{e:D-decay}
    \end{align}
\end{theorem}

\begin{remark}\label{r:uniqueness-fine-blowup}
    Observe that given the estimates \eqref{e:freq-decay}-\eqref{e:D-decay}, we obtain \emph{in any codimension} the uniqueness of the fine blow-up as defined in \cite{DMS-modp}*{Section 11.2}, with a power law decay of the rescalings of $T$ towards it. Indeed, to conclude this decay for the rescalings of the multigraph $\Fcal$ associated to $\Ncal$, the argument is exactly the same as \cite{DLS_MAMS}*{Proof of Theorem 5.3}. The only difference is that $r\mapsto\Ibf(r)$, $r\mapsto \frac{\Hbf(r)}{r^{2\alpha+1}}$ and $r\mapsto \frac{\Dbf(r)}{r^{2\alpha}}$ are not monotone but merely almost-monotone, thus our decay estimates in Theorem \ref{t:frequency-decay} are in absolute value. However, this does not yield the perturbative statement in Theorem \ref{t:main-structure}, which is specific to codimension 1; cf. Remark \ref{r:higher-codim}. To pass the decay from the rescalings of $\Fcal$ to the rescalings of $T$, we simply use the estimates of \cite{DLHMS}*{Theorem 17.24} combined with a change of variables.
\end{remark}

Before coming to the proof of Theorem \ref{t:frequency-decay}, let us use it to deduce Theorem \ref{t:main-structure} and Theorem \ref{t:isolated-modq-branch-pts}. 

\begin{proof}[Proof of Theorem \ref{t:isolated-modq-branch-pts}]
    We argue by contradiction. Suppose that there exists a point $p\in \Ffrak_Q(T)$ which is a limit point of a sequence $\{p_k\}\subset \mathfrak{F}_Q(T)$. Without loss of generality, we may assume that $p= 0$. Upon suitably rescaling $T$, Assumption \ref{a:main-m0} is satisfied, whence Corollary \ref{c:one-cm} guarantees that for all $k$ large enough $p_k = \mathbf\Phi(z_k) \in \mathbf\Phi(\Gamma)\subset\Mcal$ for $k$ sufficiently large. 
    
    If we consider $r_k := 2|z_k|$, $\Ncal_{r_k}$ converge (up to subsequence) in $W^{1,2}_\loc(B_{3/2})$ to a non-trivial $\alpha$-homogeneous Dir-minimizer $u: B_{3/2}(\pi_0) \to \Ascr_Q(\R^{\bar n})$ with $\boldsymbol\eta\circ u = 0$ and $u(0)=Q\llbracket 0\rrbracket$. Moreover, persistence of $Q$-points (cf. \cite{Sk21}*{Section 8}) guarantees that $\frac{p_k}{r_k}$ must subsequentially converge to a point $z\in B_{3/2}(\pi_0)\setminus \{0\}$ with $u(z) = Q\llbracket 0\rrbracket$. On the other hand, by Theorem \ref{t:tangents} (see \cite{DLHMSS-fine-structure}*{Theorem 3.6} for the case $\bar n =1$), we must have $I_u(z,0) = 1$ and the graph of the tangent function to $u$ at $z$ is an open book (see \cite{DLHMSS} for a precise definition). The strong $W^{1,2}_{\loc}$ convergence guarantees that the hypotheses of \cite{DMS-modp}*{Theorem 2.6} hold for the rescaling $T_{z,\rho}$ with $\rho>0$ sufficiently small. Thus, by \cite{DMS-modp}, 
    this is in contradiction with the convergence of $\frac{p_k}{r_k}$ to $z$.
\end{proof}

\begin{proof}[Proof of Theorem \ref{t:main-structure}]
    Let $r_1$ be as in Theorem \ref{t:frequency-decay}. As discussed in Remark \ref{r:uniqueness-fine-blowup}, Theorem \ref{t:frequency-decay} combined with the classification of \cite{DLHMSS-fine-structure}*{Theorem 3.6} guarantees that, up to composing with a rotation, the rescalings $T_{0, \rho}\mres \Bbf_1$ converge with a power law in two-sided $L^2$ excess to a multigraph of the form
    \[
        \Gcal(r,\theta) =\begin{cases}
            \left(\sum_i \llbracket (r,\theta) +  c_{j,i}^+\, r^{I_0}\sin(I_0\theta)\rrbracket, +1\right) & \text{if} \ (r,\theta) \in U_j^+ \\
            \left(\sum_i \llbracket (r,\theta) + c_{j,i}^-\, r^{I_0} \sin(I_0\theta)\rrbracket, -1\right) & \text{if} \ (r,\theta) \in U_j^-\,,
        \end{cases}
    \]
    for some $c_{j,i}^\pm\in\R$, $I_0 \in \N_{\geq 2}$ and
    \begin{align*}
        U_j^+ &= \left\{(r,\theta) : 0<r<r_0, \quad \tfrac{2j\pi}{I_0} < \theta \leq \tfrac{(2j+1)\pi}{I_0}\right\}\,, \\ U_j^- &= \left\{(r,\theta) : 0<r<r_0, \quad \tfrac{(2j+1)\pi}{I_0} < \theta \leq \tfrac{(2j+2)\pi}{I_0} \right\}\,.
    \end{align*}
    Namely,
    \[
        r^{-4} \int_{\Bbf_r} \dist^2(x, \spt (\Tbf_{\Gcal}))\, d\|T\| \leq C_0 r^{2\alpha}\,,
    \]
    for some $\alpha>0$ and all $r\in (0,r_2)$, for some $r_2>0$ sufficiently small. This in turn implies that for every $p\in \spt^q(T)\cap\Bbf_{r_2}\setminus \{0\}$ and $r(p) = c_0 |p|$ for $c_0>0$ chosen appropriately, we have
    \[
        r(p)^{-4} \int_{\Bbf_{r(p)}(p)} \dist^2(x, \spt (\Tbf_{\Gcal}))\, d\|T\| \leq \Lambda_0 r(p)^{2\alpha}\,.
    \]
    In particular, for each $p \in \spt^q (\Tbf_{\Gcal})\cap \Bbf_1\setminus \{0\}$, there exists $\rho(p)>0$ such that either
    \begin{itemize}
        \item[(a)] the hypotheses of \cite{DLHMSS}*{Assumption 1.8} are satisfied for the rescaling $T_{p, \rho(p)}$;
        \item[(b)] $\Theta(T_{p,\rho(p)},x) < Q$ for every $x\in \spt^q(T_{p,\rho(p)})\cap \Bbf_1$.
    \end{itemize}
    In case (a), we may apply the results of \cite{DLHMSS} to deduce that $T\mres \Bbf_{c_0\rho(p)}(p)$ is a $C^{1,\alpha}$-perturbation of $\Tbf_{\Gcal}\mres \Bbf_{c_0\rho(p)}(p)$ for some geometric constant $c_0\in (0,1)$. In case (b), we may instead apply \cite{White-regularity} to again deduce that $T\mres \Bbf_{c_0\rho(p)}(p)$ is a $C^{1,\alpha}$-perturbation of $\Tbf_{\Gcal}\mres \Bbf_{c_1\rho(p)}(p)$ for some $c_1>0$, by a map $\psi$ with $\|\psi\|_{C^{1,\alpha}} \leq C \Lambda_0$. Applying the $(5r)$-covering Lemma to the collection of balls
    \[
        \{\Bbf_{c_0\rho(p)}(p) : p \in \spt^q (\Tbf_{\Gcal})\cap \Bbf_1\setminus \{0\}\}\,
    \]
    we conclude the desired result.
\end{proof}

The key intermediate estimate towards proving Theorem \ref{t:frequency-decay} is contained in the following lemma.

\begin{lemma}\label{l:freq-key-estimate}
     Suppose that $T$ is as in Assumption \ref{a:one-cm} and suppose that $\Ibf(0) = \alpha \in \N_{\geq 2}$. There exists $\gamma(\alpha,n,Q)\in (0,1]$, $\delta(\alpha,\gamma)>0$,  and $C=C(\alpha,n,Q)>0$ such that if $|\Ibf(r) - \alpha| < \delta$ for some $r\in (0,\bar{r}]$, where $\bar{r}$ is the threshold of \cite{DLHMSS-fine-structure}*{Lemma 11.3}, then
     \begin{equation}\label{e:towards-freq-der}
         (2\alpha +\gamma)\Dbf(r) \leq \frac{r\Dbf'(r)}{2} + \frac{\alpha(\alpha+\gamma)\Hbf(r)}{r} + {\Ecal_1(r)}\,,
     \end{equation}
     where
    \begin{equation}\label{e:freq-decay-intermediate-error}
        \Ecal_1(r) := Cr^2 \Dbf'(r) + Cr\Dbf(r) + Cr^2 \int_{\partial B_r} |\boldsymbol{\eta}\circ \Ncal| + Cr^{\gamma} \int_{\mathbf\Phi(B_r)}|DN|^2\,.
    \end{equation}
     Consequently, we have
     \begin{equation}\label{e:freq-radial-der}
         \Ibf'(r) \geq \frac{2}{r}(\Ibf(r) - \alpha) (\alpha + \gamma - \Ibf(r)) {-\Ecal_2(r)}\,,
     \end{equation}
     where $\Ecal_2$ is the non-negative error given by
     \[
        \Ecal_2(r)= Cr^\gamma \Hbf(r)^{-1} \int_{\mathbf\Phi(B_r)} |DN|^2 +2\Hbf(r)^{-1}\Ecal_1(r)\,.
    \]
\end{lemma}

We begin by demonstrating that the validity of Lemma \ref{l:freq-key-estimate} implies the conclusion of Theorem \ref{t:frequency-decay}.

\begin{proof}[Proof of Theorem \ref{t:frequency-decay}]
    The argument follows the same line of reasoning as \cite{DLS_MAMS}*{Proof of Proposition 5.2} (see also \cite{DLSS3}*{Theorem 7.3}), up to well-behaved error terms. Nevertheless, we repeat it here for clarity. Fix $\delta$ and $\gamma$ as in Lemma \ref{l:freq-key-estimate}. First of all, the fact that $|\Ibf(r_0) - \alpha|< \bar\delta$ and the almost-monotonicity of $r\mapsto\Ibf(r)$ given by Corollary \ref{c:freq-monotonicity} guarantees the existence of $r_1\in (0,r_0]$ such that $\Ibf(r) \leq \alpha + \frac{\gamma}{4}$ for every $r \in (0, r_1]$, provided that $\bar\delta= \min\{\delta, \tfrac{\gamma}{8}\}$. 
Using the estimate \eqref{e:freq-radial-der}, we can then calculate, for every $\rho \in \left( 0, r_1\right]$
\begin{align*}
\frac{d}{d\rho}\left( \frac{\Ibf(\rho)-\alpha}{\rho^{\sfrac{\gamma}{2}}} \right) &= \frac{\Ibf'(\rho)}{\rho^{\sfrac{\gamma}{2}}} - \frac{\gamma}{2} \frac{\Ibf(\rho)-\alpha}{\rho^{{\sfrac{\gamma}{2}}+1}} \\
&\geq \frac{2}{\rho^{{\sfrac{\gamma}{2}}+1}}\left(\Ibf(\rho)-\alpha\right) \left(\alpha +\frac{3\gamma}{4} -\Ibf(\rho) \right) - \rho^{-{\sfrac{\gamma}{2}}} \mathcal E_2(\rho)\,.
\end{align*}
Now, for $\rho\in \left( 0, r_1\right]$ we have
$
\alpha+\frac{3\gamma}{4}-\Ibf(\rho) \geq \frac{\gamma}{2}
$,
and $\Ibf(\rho)-\alpha \geq -C \rho^\gamma$ as a consequence of Corollary \ref{c:freq-monotonicity}. This implies then that
\[
\frac{d}{d\rho}\left( \frac{\Ibf(\rho)-\alpha}{\rho^{\sfrac{\gamma}{2}}} \right) \geq -C\rho^{{\sfrac{\gamma}{2}}-1} -  \rho^{-{\sfrac{\gamma}{2}}} \mathcal E_2(\rho)\,.
\]
On the other hand, in light of a change of variables, the uniform upper bound on $\Ibf$ from Proposition \ref{p:variations}, and the estimate \cite{DLHMS}*{Theorem 17.24, (17.25)} on the average of the sheets of $N$ (summed over the Whitney cubes intersecting $\mathbf\Phi(\partial B_r)$), we have
\[
        |\Ecal_2(\rho)| \leq C \rho^{\gamma-1} + C \rho^{1+\gamma} \frac{\Dbf'(\rho)}{\Dbf(\rho)}\,,
    \]
so that in turn we obtain the estimate
\[
\frac{d}{d\rho}\left( \frac{\Ibf(\rho)-\alpha}{\rho^{\sfrac{\gamma}{2}}} \right) \geq -C\rho^{{\sfrac{\gamma}{2}}-1} -  \rho^{1+{\sfrac{\gamma}{2}}} \frac{\Dbf'(\rho)}{\Dbf(\rho)}\,.
\]
We integrate over $\rho\in\left[r,r_1\right]$ for abritrary $r \in \left( 0, r_1\right]$ to get
\[
 \frac{\Ibf(r_1)-\alpha}{r_1^{\sfrac{\gamma}{2}}} -  \frac{\Ibf(r)-\alpha}{r^{\sfrac{\gamma}{2}}} \geq -C r_1^{\sfrac{\gamma}{2}}\,,
\]
which gives the desired upper bound
\[
\Ibf(r)-\alpha \leq C r^{{\sfrac{\gamma}{2}}}\,. 
\]
This, combined with the lower bound $\Ibf(r)-\alpha \geq -C r^\gamma \geq  -C r^{{\sfrac{\gamma}{2}}}$ coming from the almost monotonicity Corollary \ref{c:freq-monotonicity} concludes the proof of \eqref{e:freq-decay}, with $\gamma$ given by $\sfrac{\gamma}{2}$. We are going to rename the exponent $\gamma$ for the sake of readability.

    Now we prove \eqref{e:H-decay}. Recalling \eqref{e:H'} and \eqref{e:outer-var-D}, we deduce that
    \[
        \frac{d}{dr}\left(\frac{\Hbf(r)}{r}\right) = \frac{2\Dbf(r)}{r} + O\left(r^{\gamma-1} \int_{\mathbf\Phi(B_r)} |DN|^2 \right)\,.
    \]
    Arguing as in \cite{DLS_MAMS}*{Proof of Proposition 5.2}, this in turn implies
    \begin{align*}
        \frac{d}{dr}\left( \log \frac{\Hbf(r)}{r^{2\alpha+1}}\right) = \frac{r}{\Hbf(r)} \frac{d}{dr}\left(\frac{\Hbf(r)}{r}\right) - \frac{2\alpha}{r} &= \frac{2}{r}(\Ibf(r) - \alpha) + O\left(r^{\gamma-1}  \right) \,,
    \end{align*}
    for every $r\in (0,r_1]$, where in the last estimate we are using \cite{DLHMS}*{Theorem 17.24}. Integrating this, we thus obtain
    \begin{align*}
        \log \frac{\Hbf(r_1)}{r_1^{2\alpha+1}} - \log \frac{\Hbf(r)}{r^{2\alpha+1}} = \log\left(\frac{\Hbf(r_1)}{r_1^{2\alpha+1}} \Big/ \frac{\Hbf(r)}{r^{2\alpha+1}} \right) = O(r^\gamma) \qquad \forall \, r\in (0,r_1] \,.
    \end{align*}
    In particular, the limit
    \[
        \lim_{r\downarrow 0} \frac{\Hbf(r)}{r^{2\alpha+1}} =: \Hbf_0 > 0
    \]
    exists, and the power law decay \eqref{e:H-decay} holds.

    Finally, let us derive \eqref{e:D-decay}. To this end, we simply observe that for $\Dbf_0 := \alpha \Hbf_0$, we have the identity
    \[
        \frac{\Dbf(r)}{r^{2\alpha}} - \Dbf_0 = (\Ibf(r) - \alpha) \frac{\Hbf(r)}{r^{2\alpha+1}} + \alpha \left(\frac{\Hbf(r)}{r^{2\alpha+1}} - \Hbf_0\right)\,.
    \]
    The combined power law decays \eqref{e:freq-decay} and \eqref{e:H-decay} demonstrated above allows us to therefore deduce the decay \eqref{e:D-decay}.
\end{proof}

The rest of this section is dedicated to the proof of Lemma \ref{l:freq-key-estimate}, which will involve comparing the Dirichlet energy of $\Ncal$ to that of $\bar\Lcal$ constructed above. We begin with the following preliminary estimates for $\bar\Lcal$. The first provides control of the Dirichlet energy of $\bar\Lcal$, while the second gives control on the average of the sheets of $\bar\Lcal$, which is needed to estimate the final error term in \eqref{e:competitor-energy-comparison}.

\begin{lemma}\label{l:competitor-estimates}
    Suppose that Assumption \ref{a:one-cm} holds. For any $\alpha\in \N_{\geq 2}$ and $\eta>0$, there exists $\delta=\delta(\alpha, \eta)>0$ such that if $\Ibf(0)=\alpha$ and for some $r\in (0,\tfrac{3}{2})$ sufficiently small we have $|\Ibf(r)-\alpha|<\delta$, then Assumption \ref{a:close-to-tangent} holds with this choice of $\eta$.
    
    In particular, there exists $\delta=\delta(\alpha)>0$, $\gamma=\gamma(\alpha,n,Q)\in(0,1]$ and $C=C(\alpha,n,Q)>0$ such that if $\Ibf(0)=\alpha$, $|\Ibf(r)-\alpha|<\delta$ for some $r\in (0,\tfrac{3}{2})$, then the competitor $\Lcal$ constructed above satisfies the estimates
    \begin{align}
        (2\alpha+\gamma)\int_{B_r} |D \bar\Lcal|^2 &\leq r \int_{\partial B_r} |\nabla_\theta \bar \Ncal|^2 + \frac{\alpha(\alpha+\gamma)}{r}\Hbf(r) \,,\label{e:competitor-Dir} \\
        \int_{B_r} |\boldsymbol\eta\circ \bar\Lcal| &\leq r \int_{\partial B_r} |\boldsymbol\eta\circ \bar \Ncal| + C r\Dbf(r) \,,\label{e:competitor-avg}
    \end{align}
    where $\nabla_\theta$ denotes the tangential gradient in $\partial B_r$.
\end{lemma}

\begin{proof}[Proof of Lemma \ref{l:competitor-estimates}]
    The first claim follows from a simple compactness argument. In fact, given $\alpha\in \N_{\geq 2}$, for any $\eta>0$, we claim that under the assumptions of the lemma, there exists $\delta=\delta(\alpha, \eta)>0$ such that both Assumption \ref{a:close-to-tangent} and the estimate
    \begin{equation}\label{e:Fourier-coeffs-projection-est}
        \sum_{i=1}^Q \sum_{\epsilon\in \{+,-\}} \left[\sum_{j=1}^\alpha \sum_{k\geq 2} (a_{i,k}^{j,\epsilon})^2 r^{2\alpha_k^{j,\epsilon,+}+1} + \sum_{j\geq \alpha+1} \sum_k (a_{i,k}^{j,\epsilon})^2 r^{2\alpha_k^{j,\epsilon,+}+1} \right] \leq \eta \Dbf(r)\,, 
    \end{equation}
    hold with this choice of $\eta$. We prove \eqref{e:Fourier-coeffs-projection-est} by contradiction. Indeed, if this is not true, then we may extract a sequence of currents $T_l$, together with a sequence $r_l \downarrow 0$, as well as corresponding center manifolds $\Mcal_l= \mathbf\Phi_l(B_{3/2})$, $\Mcal_l$-normal approximations $N_l$, and reparameterizations 
    \[
        \Ncal^{(l)}:= \frac{N_l \circ \mathbf\Phi_l(r_l \cdot)}{\Dbf_l(r_l)^{1/2}}
    \]
    over $B_{3/2} \subset \pi_0$, such that the compactness procedure in \cite{DMS-modp}*{Section 11.2} together with Theorem \cite{DMS-modp}*{Theorem 11.5} (see \cite{DLHMSS-fine-structure}*{Proposition 2.8} for the case $\bar n =1$) guarantees that
    \begin{itemize}
        \item $\Ncal^{(l)}$ converges (up to subsequence, not relabelled) in $W^{1,2}_\loc(B_{3/2})$ and locally uniformly to a non-trivial $\alpha$-homogeneous Dir-minimizer $u:B_{3/2}\to\Ascr_Q(\R^{\bar n})$ with $\boldsymbol\eta\circ u=0$ and $u(0) = Q\llbracket 0 \rrbracket$;
        \item either Assumption \ref{a:close-to-tangent} or \eqref{e:Fourier-coeffs-projection-est} fails for some $\eta>0$ and every $l$.
    \end{itemize}
    The classification of Theorem \ref{t:tangents} (see \cite{DLHMSS-fine-structure}*{Theorem 3.6} for the case $\bar n=1$) of homogeneous $\Ascr_Q$-valued Dir-minimizers combined with the explicit form of any homogeneous harmonic polynomial on $\R^2$ immediately contradicts the failure of Assumption \ref{a:close-to-tangent}. Meanwhile, the failure of \eqref{e:Fourier-coeffs-projection-est} also yields a contradiction, due to the $L^2_\loc$-convergence of $\Ncal^{(l)}$ to $u$, combined with the observation that 
    \[
         \int_{\partial B_1} |\bar \Ncal^{(l)}|^2 =\Dbf_l(r_l)^{-1}\sum_{i=1}^Q\sum_{\epsilon\in\{+,-\}}\sum_j\sum_{k} (a_{i,k}^{l,j,\epsilon})^2 r_l^{2\alpha_k^{j,\epsilon,+}+1} \,.
    \]
    {Indeed, by the very structure of $u$, all contributions in the above series coming either from regions $F^{l,\pm}_j$ with $j \geq \alpha +1$ or from modes higher than the first in the regions $F^{l,\pm}_j$ with $j \in \{1,\ldots,\alpha\}$ will vanish as $l \to \infty$}.

    Fix $\eta>0$, to be determined later, which will in turn determine $\delta$ by the above. We begin with the estimate on the Dirichlet energy. First of all, note that
    \[
         \int_{\partial B_r} |\bar \Ncal|^2 = \int_{\partial B_r} |\bar \Lcal|^2 = \sum_{i=1}^Q\sum_{\epsilon\in\{+,-\}}\sum_j\sum_{k} (a_{i,k}^{j,\epsilon})^2 r^{2\alpha_k^+ +1} \,,
    \]
    and
    \[
        \int_{\partial B_r} |\nabla_\theta \bar \Ncal|^2 = \int_{\partial B_r} |\nabla_\theta \bar \Lcal|^2 = \sum_{i=1}^Q\sum_{\epsilon\in\{+,-\}}\sum_j\sum_{k} \lambda_k^{j,\eps}(a_{i,k}^{j,\epsilon})^2 r^{2\alpha_k^+-1}\,,
    \]
    We first compute the Dirichlet energy of $\Lcal$ in the annulus $B_r\setminus \bar{B}_{\sigma r}$. Consider first $j \geq \alpha + 1$, and fix a sign $\epsilon\in\{+,-\}$. In this case, since each $\Lcal_i$ is harmonic within $V_j^\epsilon = U_j^\epsilon\cap(B_r\setminus \bar{B}_{\sigma r})$ with boundary data zero outside $F_j^\epsilon$, we have
\begin{align*}
\int_{V_j^\epsilon} |D\bar \Lcal|^2 &= \sum_{i=1}^Q \int_{F_j^\epsilon} \bar \Lcal_i\partial_{r}\bar \Lcal_i \\
&=\sum_{i=1}^Q \sum_k (a_{i,k}^{j,\epsilon})^2 r^{2\alpha_k^{j,\epsilon,+}} \left(\frac{\alpha_k^{j,\epsilon,+} \sigma^{\alpha_k^{j,\epsilon,-}} - \alpha_k^{j,\epsilon,-}\sigma^{\alpha_k^{j,\epsilon,+}}}{\sigma^{\alpha_k^{j,\epsilon,-}} - \sigma^{\alpha_k^{j,\epsilon,+}}}\right) \\
&= \sum_{i=1}^Q \sum_k (a_{i,k}^{j,\epsilon})^2 r^{2\alpha_k^{j,\epsilon,+}} \alpha_k^{j,\epsilon,+} \left(\frac{\sigma^{\alpha_k^{j,\epsilon,-}} + \sigma^{\alpha_k^{j,\epsilon,+}}}{\sigma^{\alpha_k^{j,\epsilon,-}} - \sigma^{\alpha_k^{j,\epsilon,+}}}\right)\,,
\end{align*}
where we are additionally exploiting that $\alpha_k^{j,\epsilon,+} + \alpha_k^{j,\epsilon,-} = 0$. On the other hand, for $j=1,\dots,\alpha$ and again a fixed sign $\epsilon$, combining the above computation with the fact that $|Df|^2 = |\partial_\rho f|^2 + \rho^{-2}|\nabla_\theta f|^2$ for a function $f$ in polar coordinates $(\rho,\theta)$, we have
\[
\int_{V_j^\epsilon} |D\bar \Lcal|^2= \sum_{i=1}^Q \Bigg[(a_{i,1}^{j,\epsilon})^2 \alpha_1^{j,\epsilon,+}r^{2\alpha_1^{j,\epsilon,+}}(1-\sigma^{2\alpha_1^{j,\epsilon,+}}) + \sum_{k\geq 2} (a_{i,k}^{j,\epsilon})^2 r^{2\alpha_k^{j,\epsilon,+}} \alpha_k^{j,\epsilon,+} \left(\frac{\sigma^{\alpha_k^{j,\epsilon,-}} + \sigma^{\alpha_k^{j,\epsilon,+}}}{\sigma^{\alpha_k^{j,\epsilon,-}} - \sigma^{\alpha_k^{j,\epsilon,+}}}\right)\Bigg]\,.
\]
We now compute the Dirichlet energy in $B_{\sigma r}$. Here, we are simply taking the $\alpha_1^{j,\epsilon,+}$-homogeneous extension of the boundary data on $\partial B_{\sigma r}$, so we have
\[
    \int_{B_{\sigma r}} |D\bar \Lcal|^2 = \sum_{i=1}^Q \sum_{j=1}^\alpha \sum_{\epsilon\in \{+,-\}} (a_{i,1}^{j,\epsilon})^2 (r\sigma)^{2\alpha_1^{j,\epsilon,+}}\alpha_1^{j,\epsilon,+}\,.
\]
In summary, we have
\begin{align*}
\int_{B_r} |D\bar \Lcal|^2 &= \sum_{i=1}^Q \sum_{j=1}^\alpha \sum_{\epsilon\in \{+,-\}} \Bigg[(a_{i,1}^{j,\epsilon})^2 \alpha_1^{j,\epsilon,+} r^{2\alpha_1^{j,\epsilon,+}} + \sum_{k\geq 2} (a_{i,k}^{j,\epsilon})^2 r^{2\alpha_k^{j,\epsilon,+}} \alpha_k^{j,\epsilon,+} \left(\frac{\sigma^{\alpha_k^{j,\epsilon,-}} + \sigma^{\alpha_k^{j,\epsilon,+}}}{\sigma^{\alpha_k^{j,\epsilon,-}} - \sigma^{\alpha_k^{j,\epsilon,+}}}\right)\Bigg] \\
&\qquad +\sum_{i=1}^Q \sum_{j\geq \alpha + 1} \sum_{\epsilon\in \{+,-\}}\sum_{k} (a_{i,k}^{j,\epsilon})^2 r^{2\alpha_k^{j,\epsilon,+}} \alpha_k^{j,\epsilon,+} \left(\frac{\sigma^{\alpha_k^{j,\epsilon,-}} + \sigma^{\alpha_k^{j,\epsilon,+}}}{\sigma^{\alpha_k^{j,\epsilon,-}} - \sigma^{\alpha_k^{j,\epsilon,+}}}\right)\,.
\end{align*}

Thus, we wish to find $\gamma\in (0,1]$ such that
\begin{align*}
    (2\alpha+\gamma)&\sum_{i=1}^Q \sum_{j=1}^\alpha \sum_{\epsilon\in \{+,-\}} \Bigg[(a_{i,1}^{j,\epsilon})^2 \alpha_1^{j,\epsilon,+}r^{2\alpha_1^{j,\epsilon,+}} + \sum_{k\geq 2} (a_{i,k}^{j,\epsilon})^2 r^{2\alpha_k^{j,\epsilon,+}} \alpha_k^{j,\epsilon,+} \left(\frac{\sigma^{\alpha_k^{j,\epsilon,-}} + \sigma^{\alpha_k^{j,\epsilon,+}}}{\sigma^{\alpha_k^{j,\epsilon,-}} - \sigma^{\alpha_k^{j,\epsilon,+}}}\right)\Bigg] \\
&\qquad +(2\alpha+\gamma)\sum_{i=1}^Q \sum_{j\geq \alpha + 1} \sum_{\epsilon\in \{+,-\}}\sum_{k} (a_{i,k}^{j,\epsilon})^2 r^{2\alpha_k^{j,\epsilon,+}} \alpha_k^{j,\epsilon,+} \left(\frac{\sigma^{\alpha_k^{j,\epsilon,-}} + \sigma^{\alpha_k^{j,\epsilon,+}}}{\sigma^{\alpha_k^{j,\epsilon,-}} - \sigma^{\alpha_k^{j,\epsilon,+}}}\right) \\
&\leq r\sum_{i=1}^Q\sum_{\epsilon\in\{+,-\}}\sum_j\sum_{k} \lambda_k^{j,\epsilon}(a_{i,k}^{j,\epsilon})^2 r^{2\alpha_k^{j,\epsilon,+}-1} + \frac{\alpha(\alpha+\gamma)}{r} \sum_{i=1}^Q\sum_{\epsilon\in\{+,-\}}\sum_j\sum_{k} (a_{i,k}^{j,\epsilon})^2 r^{2\alpha_k^{j,\epsilon,+} +1}\,.
\end{align*}
We will now compare coefficients. Firstly, since $\alpha, \alpha_1^{j,\epsilon,+}\in \N$, we have
\begin{align*}
    \gamma(\alpha_1^{j,\epsilon,+} - \alpha) \leq (\alpha_1^{j,\epsilon,+} - \alpha)^2
\end{align*}
for any $\gamma\in (0,1)$, which, since $\alpha_1^{j,\epsilon,+}=\sqrt{\lambda_1^{j,\epsilon}}$, is equivalent to
\begin{align*}
    (2\alpha+\gamma) \alpha_1^{j,\epsilon,+} &\leq \lambda_1^{j,\epsilon} + \alpha(\alpha+\gamma)
\end{align*}
for $j=1,\dots,\alpha$ and $k=1$. Meanwhile, for $j=1,\dots,\alpha$ and $k\geq 2$, the fact that 
\begin{equation*}
    \frac{\sigma^{\alpha_k^{j,\epsilon,-}} + \sigma^{\alpha_k^{j,\epsilon,+}}}{\sigma^{\alpha_k^{j,\epsilon,-}} - \sigma^{\alpha_k^{j,\epsilon,+}}} = \frac{1+ \sigma^{2\sqrt{\lambda_k^{j,\epsilon}}}}{1-\sigma^{2\sqrt{\lambda_k^{j,\epsilon}}}} = 1 +\frac{2\sigma^{2\sqrt{\lambda_k^{j,\epsilon}}}}{1-\sigma^{2\sqrt{\lambda_k^{j,\epsilon}}}} \qquad \forall j, k, \epsilon\,,
\end{equation*}
tells us that 
\begin{align*}
    (2\alpha+\gamma) \alpha_k^{j,\epsilon,+} \left(\frac{\sigma^{\alpha_k^{j,\epsilon,-}} + \sigma^{\alpha_k^{j,\epsilon,+}}}{\sigma^{\alpha_k^{j,\epsilon,-}} - \sigma^{\alpha_k^{j,\epsilon,+}}}\right) \leq (2\alpha+\gamma) \alpha_k^{j,\epsilon,+} + \frac{2(2\alpha+\gamma) \sqrt{\lambda_k^{j,\epsilon}} \sigma^{2\sqrt{\lambda_k^{j,\epsilon}}}}{1-\sigma^{2\sqrt{\lambda_k^{j,\epsilon}}}} \,,
\end{align*}
for each such $j,k$. For a choice of $\eta(\alpha,\gamma,\varsigma)\in (0,1)$ small enough (which in turn determines $\sigma$ depending on the same parameters), we may thus ensure that
\[
    (2\alpha+\gamma) \alpha_k^{j,\epsilon,+} \left(\frac{\sigma^{\alpha_k^{j,\epsilon,-}} + \sigma^{\alpha_k^{j,\epsilon,+}}}{\sigma^{\alpha_k^{j,\epsilon,-}} - \sigma^{\alpha_k^{j,\epsilon,+}}}\right) \leq (2\alpha+\gamma) \alpha_k^{j,\epsilon,+} + \varsigma\,.
\]
Similarly, we can also ensure that for $j\geq \alpha + 1$ and all $k\geq 1$ we have
\begin{align*}
    (2\alpha+\gamma) \alpha_k^{j,\epsilon,+} \left(\frac{\sigma^{\alpha_k^{j,\epsilon,-}} + \sigma^{\alpha_k^{j,\epsilon,+}}}{\sigma^{\alpha_k^{j,\epsilon,-}} - \sigma^{\alpha_k^{j,\epsilon,+}}}\right) \leq (2\alpha+\gamma) \alpha_k^{j,\epsilon,+} + \varsigma \,.
\end{align*}
Properties (ii) and (iii) then guarantee the existence of $\gamma>0$ such that
\[
    \varsigma + \gamma \left(\sqrt{\lambda_k^{j,\epsilon}}-\alpha\right) \leq \left(\sqrt{\lambda_k^{j,\epsilon}} - \alpha\right)^2\,,
\]
which is equivalent to
\[
    (2\alpha +\gamma)\alpha_k^{j,\epsilon,+} +\varsigma \leq \lambda_k^{j,\epsilon} +\alpha(\alpha+\gamma)\,,
\]
for $j=1,\dots,\alpha$ and $k\geq 2$, and for $j\geq \alpha + 1$ and $k\geq 1$. This yields the desired estimate \eqref{e:competitor-Dir} on the Dirichlet energy of $\bar \Lcal$.

We now turn to the estimate \eqref{e:competitor-avg} on the average of the sheets of $\bar \Lcal$. We again subdivide $B_r$; we begin with the conical regions $V_j^\epsilon$ in the annulus $(B_r\setminus \bar{B}_{\sigma r})$ as above. First, consider $j \geq \alpha + 1$ and fix a sign $\epsilon \in \{+,-\}$. In this case we have
    \begin{equation*}
        \int_{V_j^\epsilon} |\boldsymbol\eta\circ \bar \Lcal| \leq \frac{1}{Q} \sum_{i=1}^Q \int_{\sigma r}^r \int_{r^{-1} F_j^\epsilon} \left|\sum_{k} a_{i,k} \left(A_k^{j,\epsilon} \rho^{\alpha_k^{j,\epsilon,+} +1} - B_k^{j,\epsilon} \rho^{\alpha_k^{j,\epsilon,-} + 1}\right)  \vphi_k^{j,\epsilon}(\theta) \right| \, d\theta \, d\rho \,.
    \end{equation*}
Notice now that for every $\sigma r \leq \rho \leq r$ it holds
\[
A_k^{j,\epsilon} \rho^{\alpha_k^{j,\epsilon,+} +1} - B_k^{j,\epsilon} \rho^{\alpha_k^{j,\epsilon,-} + 1}=\rho^{1-\alpha_k^{j,\epsilon,+}} \,\frac{\rho^{2\alpha_k^{j,\epsilon,+}}-(\sigma r)^{2\alpha_k^{j,\epsilon,+}}}{1-\sigma^{2\alpha_k^{j,\epsilon,+}}} \geq 0\,.
\]
Moreover, upon further decreasing $\eta>0$ if necessary, we can assume in what follows that $5\leq \lambda_k^{j,\epsilon}$ (recall (i)-(iii)), so that we can integrate in the $\rho$ variable and further estimate
\begin{align*}
&\qquad\int_{V_j^\epsilon} |\boldsymbol\eta\circ \bar \Lcal| \\
&\leq \frac{1}{Q} \sum_{i=1}^Q \sum_k \left( \frac{A_k^{j,\epsilon} r^{\alpha_k^{j,\epsilon,+}+2}(1-\sigma^{\alpha_k^{j,\epsilon,+} + 2})}{\alpha_k^{j,\epsilon,+} + 2} - \frac{B_k^{j,\epsilon} r^{\alpha_k^{j,\epsilon,-}+2}(1-\sigma^{\alpha_k^{j,\epsilon,-} + 2})}{\alpha_k^{j,\epsilon,-}+2}\right) \int_{r^{-1}F_j^\epsilon}  \left| a_{i,k} \varphi_k^{j,\epsilon}(\theta) \right|\,d\theta\,.
\end{align*}
    
    Meanwhile, for $j\leq \alpha$ we have
    \begin{align*}
        &\int_{V_j^\epsilon} |\boldsymbol\eta\circ \bar \Lcal| \leq \frac{1}{Q} \sum_{i=1}^Q \int_{\sigma r}^r \int_{r^{-1} F_j^\epsilon} \left|a_{i,1}^{j,\epsilon} \rho^{\alpha_1^{j,\epsilon,+}+1}\vphi_1^{j,\epsilon}(\theta) + \sum_{k\geq 2} a_{i,k}^{j,\epsilon}\left(A_k^{j,\epsilon} \rho^{\alpha_k^{j,\epsilon,+} +1} - B_k^{j,\epsilon} \rho^{\alpha_k^{j,\epsilon,-} + 1}\right)\vphi_k(\theta) \right| \\
        &\leq \frac{1}{Q} \sum_{i=1}^Q \int_{r^{-1}F_j^\epsilon} \frac{r^{\alpha_1^{j,\epsilon,+} +2}(1-\sigma^{\alpha_1^{j,\epsilon,+}+2})}{\alpha_1^{j,\epsilon,+}+2} \left|a_{i,1}^{j,\epsilon}\vphi_1^{j,\epsilon}(\theta) \right|\\
        &\quad\qquad+ \sum_{k\geq 2} \left(\frac{A_k^{j,\epsilon}r^{\alpha_k^{j,\epsilon,+}+2}(1-\sigma^{\alpha_k^{j,\epsilon,+} +2})}{\alpha_k^{j,\epsilon,+} +2} - \frac{B_k^{j,\epsilon}r^{\alpha_k^{j,\epsilon,-}+2}(1-\sigma^{\alpha_k^{j,\epsilon,-} +2})}{\alpha_k^{j,\epsilon,-} +2} \right)\left|a_{i,k}^{j,\epsilon}\vphi_k^{j,\epsilon}(\theta) \right| \, d\theta \\
        &\leq \frac{r^{\alpha_1^{j,\epsilon,+}+2}(1-\sigma^{\alpha_1^{j,\epsilon,+}+2})}{Q(\alpha_1^{j,\epsilon,+}+2)} \sum_{i=1}^Q \int_{r^{-1} F_j^\epsilon} |a_{i,1}^{j,\epsilon}\vphi_1^{j,\epsilon}(\theta)|\,d\theta \\
        &\quad+ \frac{1}{Q}\sum_{k\geq 2}\int_{r^{-1}F_j^\epsilon}\left(\frac{A_k^{j,\epsilon}r^{\alpha_k^{j,\epsilon,+}+2}(1-\sigma^{\alpha_k^{j,\epsilon,+} +2})}{\alpha_k^{j,\epsilon,+} +2} - \frac{B_k^{j,\epsilon}r^{\alpha_k^{j,\epsilon,-}+2}(1-\sigma^{\alpha_k^{j,\epsilon,-} +2})}{\alpha_k^{j,\epsilon,-} +2} \right)\left|a_{i,k}^{j,\epsilon}\vphi_k^{j,\epsilon}(\theta) \right| \, d\theta
    \end{align*}
    Finally, in $B_{\sigma r}$ we have
    \begin{align*}
        \int_{B_{\sigma r}} |\boldsymbol\eta\circ \Lcal| &\leq \frac{(\sigma r)^{\alpha_1^{j,\epsilon,+}+2}}{Q(\alpha_1^{j,\epsilon,+}+2)} \sum_{i=1}^Q \sum_{j=1}^\alpha \sum_{\epsilon \in \{+,-\}}\int_{r^{-1} F_j^\epsilon} \left|a_{i,1}^{j,\epsilon} \vphi_1^{j,\epsilon}(\theta)\right|\, d\theta
    \end{align*}
    Combining these estimates with \eqref{e:Fourier-coeffs-projection-est} which implies
    \begin{align*}
        \int_{\partial B_r} |\boldsymbol\eta\circ \bar\Ncal| &= \int_{\partial B_r} |\boldsymbol\eta\circ \bar\Lcal|= \frac{1}{Q}\sum_{i=1}^Q \sum_j \sum_{\epsilon\in \{+,-\}} r^{\alpha_k^{j,\epsilon,+}+1}\int_{r^{-1} F_j^\epsilon} \left|\sum_k 
        a_{i,k}^{j,\epsilon}\vphi_k^{j,\epsilon}(\theta)\right|\,d\theta \\
        &= \frac{1}{Q}\sum_{i=1}^Q \sum_{j=1}^\alpha \sum_{\epsilon\in \{+,-\}} r^{\alpha_1^{j,\epsilon,+}+1}\int_{r^{-1}F_j^\epsilon} \left|
        a_{i,1}^{j,\epsilon}\vphi_1^{j,\epsilon}(\theta)\right|\,d\theta + O\big(\eta \Dbf(r)\big)
    \end{align*}
    we therefore obtain
    \begin{align*}
        \int_{B_r} |\boldsymbol\eta\circ\bar\Lcal| &\leq r\int_{\partial B_r} |\boldsymbol\eta\circ \bar\Ncal| +C\eta r\Dbf(r) \\
        &\qquad+ \frac{1}{Q}\sum_{i=1}^Q\sum_{j=1}^\alpha \sum_{\epsilon\in\{+,-\}}\sum_{k\geq 2} \left(\Lambda^{j,\epsilon,+}_k - \Lambda^{j,\epsilon,-}_k \right)\left|a_{i,k}^{j,\epsilon}\vphi_k^{j,\epsilon}(\theta) \right| \\
        &\qquad+ \frac{1}{Q}\sum_{i=1}^Q\sum_{j\geq \alpha+1}\sum_{\epsilon\in\{+,-\}}\sum_{k} \left(\Lambda^{j,\epsilon,+}_k - \Lambda^{j,\epsilon,-}_k \right)\left|a_{i,k}^{j,\epsilon}\vphi_k^{j,\epsilon}(\theta) \right|\,,
    \end{align*}
    where
    \[
        \Lambda^{j,\epsilon,+}_k := \frac{A_k^{j,\epsilon}r^{\alpha_k^{j,\epsilon,+}+2}(1-\sigma^{\alpha_k^{j,\epsilon,+} +2})}{\alpha_k^{j,\epsilon,+} +2} \qquad \text{and} \qquad \Lambda^{j,\epsilon,-}_k := \frac{B_k^{j,\epsilon}r^{\alpha_k^{j,\epsilon,-}+2}(1-\sigma^{\alpha_k^{j,\epsilon,-} +2})}{\alpha_k^{j,\epsilon,-} +2}\,,
    \]
    and where $\Lambda^{j,\epsilon,+}_k-\Lambda^{j,\epsilon,-}_k \geq 0$, as already observed. It therefore remains to treat the error terms on the right-hand side. In light of \eqref{e:Fourier-coeffs-projection-est} and the definitions of $A_k$, $B_k$, and $\alpha_k^{j,\epsilon,\pm}$, we simply need to verify that 
    \begin{equation}\label{e:coeffs-bdd}
     (0\leq) \frac{\sigma^{\alpha_k^-}}{\sigma^{\alpha_k^-} - \sigma^{\alpha_k^+}}
      \frac{1-\sigma^{\alpha_k^{j,\epsilon,+} +2}}{\alpha_k^{j,\epsilon,+} +2} - \frac{\sigma^{\alpha_k^+}}{\sigma^{\alpha_k^-} - \sigma^{\alpha_k^+}}\frac{1-\sigma^{\alpha_k^{j,\epsilon,-} +2}}{\alpha_k^{j,\epsilon,-} +2} \leq C\,,
    \end{equation}
    for a universal constant $C>0$, in order to conclude the estimate \eqref{e:competitor-avg}. It is now elementary algebra to check that the left-hand side of the above inequality equals
    \[
    {\rm RHS}=\frac{1}{1-\sigma^{2\alpha_k^{j,\epsilon,+}}} \,\frac{\alpha_k^{j,\epsilon,+}\left(1-2\sigma^{\alpha_k^{j,\epsilon,+}+2}+\sigma^{2\alpha_k^{j,\epsilon,+}}\right) + 2\left(\sigma^{2\alpha_k^{j,\epsilon,+}}-1 \right)}{\lambda_k^{j,\epsilon}-4}\,,
    \]
    and since $\lambda_k^{j,\epsilon}\geq 5$ (and thus $\alpha_k^{j,\epsilon,+}\geq 2$) this can be estimated from above as
    \[
    {\rm RHS}\leq \frac{2\alpha_k^{j,\epsilon,+}}{\lambda_k^{j,\epsilon}-4} \frac{1-\sigma^{\alpha_k^{j,\epsilon,+}+2}}{1-\sigma^{2\alpha_k^{j,\epsilon,+}}} \leq \frac{\sqrt{\lambda_k^{j,\epsilon}}}{\lambda_k^{j,\epsilon}-4} \leq C\,,
    \]
    thus concluding the proof.
\end{proof}

Armed with Lemma \ref{l:competitor-estimates} and Proposition \ref{p:almost-min}, we are now in a position to prove Lemma \ref{l:freq-key-estimate}.

\begin{proof}[Proof of Lemma \ref{l:freq-key-estimate}]
    We will closely follow the computations in \cite{DLS_MAMS}*{Proof of Proposition 5.2} and \cite{DLSS3}*{Sections 4-7}. 

    Fix $\delta(\alpha)$ such that the conclusions \eqref{e:competitor-Dir} and \eqref{e:competitor-avg} of Lemma \ref{l:competitor-estimates} hold. We will use the preceding notation for the competitor $\Lcal$ and the eigenfunction expansion of $\Ncal|_{\partial B_r}$. Fix $\gamma > 0$, to be determined later with the desired dependencies. First of all, Proposition \ref{p:almost-min}, the estimate \eqref{e:competitor-Dir} (simplifying error terms) and \eqref{e:inner-var-D'} together yield
    \begin{align*}
        (2\alpha &+ \gamma) \Dbf(r) \\
        &\leq (1+Cr) \left[ r\Gbf(r) + \frac{\alpha(\alpha+\gamma)}{r}\Hbf(r)\right] + Cr\Dbf(r) + Cr^2\int_{\partial B_r} |\boldsymbol{\eta}\circ \Ncal|\\
        &\leq \frac{r\Dbf'(r)}{2} + \frac{\alpha(\alpha+\gamma)}{r}\Hbf(r) + \Ecal_1(r)\,,
    \end{align*}
    where, in light of the uniform upper bound for $\Ibf$ in Corollary \ref{c:freq-monotonicity} (see \cite{DLHMSS-fine-structure}*{Lemma 11.3} for the case $\bar n=1$), we obtain \eqref{e:towards-freq-der} with $\Ecal_1(r)$ as claimed in \eqref{e:freq-decay-intermediate-error}.
    
    Now let us demonstrate \eqref{e:freq-radial-der}. Thanks to \eqref{e:H'}, \eqref{e:outer-var-D} and again the uniform upper bound on $\Ibf(r)$, we have
    \begin{align*}
        \Ibf'(r) &= \frac{\Dbf(r)}{\Hbf(r)} + \frac{r\Dbf'(r)}{\Hbf(r)} - \frac{r\Dbf(r)\Hbf'(r)}{\Hbf(r)^2} \\
        &= \frac{r\Dbf'(r)}{\Hbf(r)} - \frac{2r\Dbf(r)\Ebf(r)}{\Hbf(r)^2} \\
        &= \frac{r\Dbf'(r)}{\Hbf(r)} - \frac{2\Ibf(r)^2}{r} + O\left(\Hbf(r)^{-1} r^\gamma \int_{\mathbf\Phi(B_r)} |DN|^2\right)\,.
    \end{align*}
    Combining this with \eqref{e:towards-freq-der}, we arrive at
    \begin{align*}
        \Ibf'(r) &\geq 2(2\alpha+\gamma)\frac{\Dbf(r)}{\Hbf(r)} -\frac{2\alpha(\alpha+\gamma)}{r} - 2\Hbf(r)^{-1} \Ecal_1(r) - \frac{2\Ibf(r)^2}{r} + O\left(\Hbf(r)^{-1} r^\gamma \int_{\mathbf\Phi(B_r)} |DN|^2\right) \\
        &=\frac{2}{r}(\Ibf(r)-\alpha)(\alpha+\gamma-\Ibf(r)) -\Ecal_2(r)\,,
    \end{align*}
\end{proof}
as desired.

\bibliography{references}

\begin{bibdiv}
\begin{biblist}

\bib{Almgren_regularity}{book}{
      author={Almgren, Frederick~J},
       title={Almgren's big regularity paper: Q-valued functions minimizing
  {D}irichlet's integral and the regularity of area-minimizing rectifiable
  currents up to codimension 2},
   publisher={World scientific},
        date={2000},
      volume={1},
}

\bib{DLHMSS}{article}{
      author={De~Lellis, Camillo},
      author={Hirsch, Jonas},
      author={Marchese, Andrea},
      author={Spolaor, Luca},
      author={Stuvard, Salvatore},
       title={Area minimizing hypersurfaces modulo $ p $: a geometric
  free-boundary problem},
        date={2021},
     journal={preprint arXiv:2105.08135},
}

\bib{DLHMSS-fine-structure}{article}{
      author={De~Lellis, Camillo},
      author={Hirsch, Jonas},
      author={Marchese, Andrea},
      author={Spolaor, Luca},
      author={Stuvard, Salvatore},
       title={Fine structure of the singular set of area minimizing
  hypersurfaces modulo $p$},
        date={2022},
     journal={preprint arXiv:2201.10204},
}

\bib{DLHMSS-excess-decay}{article}{
      author={De~Lellis, Camillo},
      author={Hirsch, Jonas},
      author={Marchese, Andrea},
      author={Spolaor, Luca},
      author={Stuvard, Salvatore},
       title={Excess decay for minimizing hypercurrents {${\rm mod}\, 2 Q$}},
        date={2024},
        ISSN={0362-546X},
     journal={Nonlinear Anal.},
      volume={247},
       pages={Paper No. 113606, 47},
         url={https://doi-org.pros1.lib.unimi.it/10.1016/j.na.2024.113606},
}

\bib{DLHMS}{article}{
      author={De~Lellis, Camillo},
      author={Hirsch, Jonas},
      author={Marchese, Andrea},
      author={Stuvard, Salvatore},
       title={Regularity of area minimizing currents mod p},
        date={2020},
     journal={Geom. Funct. Anal.},
      volume={30},
      number={5},
       pages={1224\ndash 1336},
}

\bib{DLHMS_linear}{article}{
      author={De~Lellis, Camillo},
      author={Hirsch, Jonas},
      author={Marchese, Andrea},
      author={Stuvard, Salvatore},
       title={Area-minimizing currents {${\rm mod}\,2Q$}: linear regularity
  theory},
        date={2022},
        ISSN={0010-3640,1097-0312},
     journal={Comm. Pure Appl. Math.},
      volume={75},
      number={1},
       pages={83\ndash 127},
         url={https://doi.org/10.1002/cpa.21964},
}

\bib{DMS-modp}{article}{
      author={De~Lellis, Camillo},
      author={Minter, Paul},
      author={Skorobogatova, Anna},
       title={Fine structure of singularities in area-minimizing currents
  mod$(q)$},
        date={2024},
     journal={preprint arXiv:2403.15889},
}

\bib{DLSk1}{article}{
      author={De~Lellis, Camillo},
      author={Skorobogatova, Anna},
       title={The fine structure of the singular set of area-minimizing
  integral currents {I}: the singularity degree of flat singular points},
        date={2023},
     journal={preprint arXiv:2304.11552},
}

\bib{DLSk2}{article}{
      author={De~Lellis, Camillo},
      author={Skorobogatova, Anna},
       title={The fine structure of the singular set of area-minimizing
  integral currents {II}: rectifiability of flat singular points with
  singularity degree larger than $1$},
        date={2023},
     journal={preprint arXiv:2304.11555},
}

\bib{DLS_MAMS}{article}{
      author={De~Lellis, Camillo},
      author={Spadaro, Emanuele},
       title={{$Q$}-valued functions revisited},
        date={2011},
        ISSN={0065-9266},
     journal={Mem. Amer. Math. Soc.},
      volume={211},
      number={991},
       pages={vi+79},
         url={https://doi.org/10.1090/S0065-9266-10-00607-1},
}

\bib{DLS16centermfld}{article}{
      author={De~Lellis, Camillo},
      author={Spadaro, Emanuele},
       title={Regularity of area minimizing currents {II}: center manifold},
        date={2016},
     journal={Ann. of Math.(2)},
       pages={499\ndash 575},
}

\bib{DLS16blowup}{article}{
      author={De~Lellis, Camillo},
      author={Spadaro, Emanuele},
       title={Regularity of area minimizing currents {III}: blow-up},
        date={2016},
     journal={Ann. of Math.(2)},
       pages={577\ndash 617},
}

\bib{DLSS2}{article}{
      author={De~Lellis, Camillo},
      author={Spadaro, Emanuele},
      author={Spolaor, Luca},
       title={Regularity theory for 2-dimensional almost minimal currents {II}:
  branched center manifold},
        date={2017},
     journal={Ann. PDE},
      volume={3},
       pages={1\ndash 85},
}

\bib{DLSS1}{article}{
      author={De~Lellis, Camillo},
      author={Spadaro, Emanuele},
      author={Spolaor, Luca},
       title={Regularity theory for 2-dimensional almost minimal currents {I}:
  Lipschitz approximation},
        date={2018},
     journal={Trans. Amer. Math. Soc.},
      volume={370},
      number={3},
       pages={1783\ndash 1801},
}

\bib{DLSS3}{article}{
      author={De~Lellis, Camillo},
      author={Spadaro, Emanuele},
      author={Spolaor, Luca},
       title={Regularity theory for $2 $-dimensional almost minimal currents
  iii: Blowup},
        date={2020},
     journal={J. Differential Geom.},
      volume={116},
      number={1},
       pages={125\ndash 185},
}

\bib{Federer1970}{article}{
      author={Federer, Herbert},
       title={The singular sets of area minimizing rectifiable currents with
  codimension one and of area minimizing flat chains modulo two with arbitrary
  codimension},
        date={1970},
     journal={Bull. Amer. Math. Soc},
      volume={76},
       pages={767\ndash 771},
         url={https://doi.org/10.1090/S0002-9904-1970-12542-3},
}

\bib{HSSS}{article}{
      author={Hirsch, Jonas},
      author={Skorobogatova, Anna},
      author={Spolaor, Luca},
      author={Stuvard, Salvatore},
       title={Forthcoming},
}

\bib{KW-density-2}{article}{
      author={Krummel, Brian},
      author={Wickramasekera, Neshan},
       title={Fine properties of branch point singularities: stationary
  two-valued graphs and stable minimal hypersurfaces near points of density $<
  3$},
        date={2021},
     journal={preprint arXiv:2111.12246},
}

\bib{Liu-mod-p}{article}{
      author={Liu, Zhenhua},
       title={Homologically area-minimizing surfaces mod $v$ have at worst
  codimension 2 singular sets asymptotically},
        date={2024},
     journal={preprint arXiv:2401.18074},
}

\bib{Min2}{article}{
      author={Minter, Paul},
       title={The structure of stable codimension one integral varifolds near
  classical cones of density {Q}+ 1/2},
        date={2024},
     journal={Calc. Var. Partial Differential Equations},
      volume={63},
      number={1},
       pages={5},
}

\bib{MW}{article}{
      author={Minter, Paul},
      author={Wickramasekera, Neshan},
       title={A structure theory for stable codimension 1 integral varifolds
  with applications to area minimising hypersurfaces mod $p$},
        date={2024},
     journal={J. Amer. Math. Soc.},
      volume={37},
      number={3},
       pages={861\ndash 927},
}

\bib{Simon_cylindrical}{article}{
      author={Simon, Leon},
       title={Cylindrical tangent cones and the singular set of minimal
  submanifolds},
        date={1993},
     journal={J. Differential Geom.},
      volume={38},
      number={3},
       pages={585\ndash 652},
}

\bib{Sk21}{article}{
      author={Skorobogatova, Anna},
       title={An upper {M}inkowski bound for the interior singular set of area
  minimizing currents},
     journal={Comm. Pure Appl. Math. to appear},
        note={available on arXiv:2108.00418},
}

\bib{Sk-modp}{article}{
      author={Skorobogatova, Anna},
       title={Rectifiability of flat singular points for area-minimizing {${\rm
  mod}(2Q)$} hypercurrents},
        date={2024},
        ISSN={1073-7928},
     journal={Int. Math. Res. Not. IMRN},
      number={11},
       pages={9237\ndash 9255},
         url={https://doi-org.pros2.lib.unimi.it/10.1093/imrn/rnae024},
}

\bib{JTaylor76}{article}{
      author={Taylor, Jean~E.},
       title={Regularity of the singular sets of two-dimensional
  area-minimizing flat chains modulo {$3$} in {$R\sp{3}$}},
        date={1973},
        ISSN={0020-9910,1432-1297},
     journal={Invent. Math.},
      volume={22},
       pages={119\ndash 159},
         url={https://doi.org/10.1007/BF01392299},
}

\bib{White-mod4}{article}{
      author={White, Brian},
       title={The structure of minimizing hypersurfaces mod {$4$}},
        date={1979},
        ISSN={0020-9910,1432-1297},
     journal={Invent. Math.},
      volume={53},
      number={1},
       pages={45\ndash 58},
         url={https://doi.org/10.1007/BF01403190},
}

\bib{White-regularity}{incollection}{
      author={White, Brian},
       title={A regularity theorem for minimizing hypersurfaces modulo {$p$}},
        date={1986},
   booktitle={Geometric measure theory and the calculus of variations
  ({A}rcata, {C}alif., 1984)},
      series={Proc. Sympos. Pure Math.},
      volume={44},
   publisher={Amer. Math. Soc., Providence, RI},
       pages={413\ndash 427},
         url={https://doi.org/10.1090/pspum/044/840290},
}

\bib{W14_annals}{article}{
      author={Wickramasekera, Neshan},
       title={A general regularity theory for stable codimension 1 integral
  varifolds},
        date={2014},
     journal={Ann. of Math.(2)},
       pages={843\ndash 1007},
}

\end{biblist}
\end{bibdiv}
\bibliographystyle{plain}

\end{document}